\documentclass{svjour3}
\smartqed

\usepackage{diagbox}
\usepackage{epic,eepic}
\usepackage{multirow}
\usepackage{tikz}
\usepackage{amsfonts,amssymb}
\usepackage{caption}
\usepackage{bm}
\usepackage[bf,SL,BF]{subfigure}
\usepackage{epstopdf}
\usepackage{graphicx,ifthen,latexsym,mathrsfs,multicol}
\usepackage{amsmath}
\usepackage{epstopdf}
\usepackage{bbding}
\usepackage{mathptmx}
\usepackage{algorithm}
\usepackage{algorithmic}
\usepackage{color}

\usepackage{booktabs}
\usepackage{mathtools}

\numberwithin{equation}{section}
\numberwithin{theorem}{section}
\numberwithin{lemma}{section}
\numberwithin{remark}{section}

\begin{document}

\title{Correction of high-order  $L_k$ approximation for subdiffusion}

\author{Jiankang Shi    \and
         Minghua Chen    \and
         Yubin Yan        \and
          Jianxiong Cao
}

\institute{J. Shi \and M. Chen (\Envelope)   \at
              School of Mathematics and Statistics, Gansu Key Laboratory of Applied Mathematics and Complex Systems, Lanzhou University, Lanzhou 730000, P.R. China\\
email:chenmh@lzu.edu.cn;  shijk17@lzu.edu.cn\\
   \and
         Y. Yan \at
              epartment of Mathematical and Physical Sciences
Fac. of Science and Engineering, University of Chester
Thornton Science Park, Pool Lane, Ince, CH2 4NU, UK\\
              email: y.yan@chester.ac.uk
           \and
         J. Cao \at
              School of Sciences, Lanzhou University of Technology, Lanzhou 730000, P.R. China\\
              email: caojianxiong2007@126.com
}

\date{Received: date / Accepted: date}

\maketitle

\begin{abstract}

The subdiffusion equations with a Caputo fractional  derivative of order $\alpha \in (0,1)$   arise in a wide variety of  practical problems,
which is describing the transport processes, in the force-free limit, slower than Brownian diffusion.
In this work, we derive the correction schemes of the Lagrange interpolation  with degree $k$ ($k\leq 6$)  convolution quadrature, called  $L_k$ approximation,   for the subdiffusion,
which are easy to implement on variable grids.
The key step of designing  correction algorithm  is to calculate the explicit form of the coefficients of  $L_k$ approximation by the polylogarithm function  or  Bose-Einstein integral.
To construct a  $\tau_8$ approximation of  Bose-Einstein integral,  the desired  $(k+1-\alpha)$th-order convergence rate can be proved
for the correction $L_k$ scheme with nonsmooth data, which is  higher  than $k$th-order BDF$k$ method  in [Jin, Li, and Zhou,
SIAM J. Sci. Comput., 39 (2017), A3129--A3152; Shi  and Chen, J. Sci. Comput., (2020) 85:28].   The numerical experiments with spectral method are given to illustrate theoretical results.

\end{abstract}

\section{Introduction}\label{Se:intro}
The subdiffusion equations  are a type of partial differential equations describing the transport processes,
which are, in the force-free limit, slower than Brownian diffusion.  Many application problems can be modeled by subdiffusion, such as
 underground environmental problems, transport in turbulent plasma, bacterial motion
transport in micelle systems and in heterogeneous rocks, porous systems, dynamics of a bead in
a polymeric network \cite{Metzler:00}. In this work, we study  the high order time discretization schemes by the Lagrange interpolation of degree $k\leq 6$,
called  $L_k$ approximation,   for  solving  the  subdiffusion, whose prototype is \cite{Podlubny:1999},  for  $0<\alpha<1$
\begin{equation} \label{fee}
 \left \{
\begin{split}
& ^C_0{D}^{\alpha}_{t} u(t) - A u(t)= f(t), \quad 0<t<T, \\
& u(0) = v.
\end{split}
\right.
\end{equation}
Here  $f$ is a given function, the operator $A=\Delta$  denotes Laplacian on a polyhedral domain $\Omega \subset \mathbb{R}^d$, $d=1,2,3$. The operator $^C_0{D}^{\alpha}_{t}$ denotes the Caputo fractional derivative, namely,
\begin{equation}\label{Cfd}
  ^C_0{D}^{\alpha}_{t}u(t) = \frac{1}{\Gamma(1-\alpha)} \int^{t}_{0} {(t-s)^{-\alpha}u'(s)} ds.
\end{equation}

Because of the nonlocal properties of Caputo fractional  derivative \eqref{Cfd}, the correction of higher order $L_k$ approximation   play
a more important role in discretizing Caputo fractional derivatives than classical ones  \cite{CD:2014}. The striking feature
is that higher order $L_k$ approximation of nonlocal  operators can keep the same computation cost with
$L_1$ schemes  but greatly improve the accuracy. In recent years, there are some important progress has been made for numerically solving the subdiffusion.
For example, under the smooth assumption,
Lin and Xu  developed the $L_1$ schemes for the subdiffusion and the optimal convergence rate with $\mathcal{O}(\tau^{2-\alpha})$ has been obtained  in \cite{LX:2007}.
Gao et al. \cite{GSZ:2014} showed that the desired $\mathcal{O}(\tau^{3-\alpha})$ convergence rate can be achieved for $L_2$ approximation by some numerical simulation.
Later on,  Lv and Xu    establish stability and convergence analysis of $L_2$ approximation in \cite{LX:2016}.
Cao and Li at al. provided a high-order $\mathcal{O}(\tau^{k+1-\alpha})$, $k\leq 5$, for subdiffusion  with  $L_k$ approximation \cite{CLC:2015,LCL:2016}, where the convergence
analysis also  remains to be proved.

It is well known that   the smoothness of all the data of \eqref{fee} do not imply the smoothness of the solution $u$.
For example,  the following estimate holds if  $f=0$  \cite{SY:11,SC:2020}, namely,
$$\|  ^C_0{D}^{\alpha}_{t}u(t)\|_{L^2(\Omega)}\leq ct^{-\alpha}\| u_0\|_{L^2(\Omega)},$$
which reduces to a  parabolic problem $\|\partial_t u(t)\|_{L^2(\Omega)}\leq ct^{-1}\| u_0\|_{L^2(\Omega)}$ if $\alpha=1$ \cite[p.\,39]{Thomee:2006}.
It implies  that $u$ has an initial layer  at $t\rightarrow 0^{+}$ \cite{SOG:2017}.
In another word, the high-order convergence rates may not hold for nonsmooth data.
Hence, the efficiently solving the subdiffusion naturally becomes an urgent topic.
Luckily, there are already two predominant  discretization techniques in time direction to restore the desired convergence rate for nonsmooth data.
The first type is that the nonuniform time meshes/graded meshes are employed to compensate for the singularity of
the continuous solution near $t=0$. For example, Stynes et al.  capture the singularity of the solution for  subdiffusion \eqref{fee} and the optimal convergence rate
with $\mathcal{O}(\tau^{2-\alpha})$ of the time discretization schemes can be restored \cite{SOG:2017}.
The corresponding theoretical and algorithm can also be extended to  the Caputo fractional substantial derivative equation \cite{CJB:2021}.
Using a nonstandard set of basis functions,  Kopteva \cite{Kopteva:2021} provide an $L_2$ approximation for subdiffusion, which proved and restored  the optimal order $3-\alpha$ on graded meshes.

The second  type is that,  based on correction of high-order  BDF$k$  or   $L_k$ approximation, the desired  high-order convergence rates can be restored even for nonsmooth initial data.
For example, Lubich et al.  provided the corrected BDF$2$ and proved the optimal convergence orders for an evolution equaiton with a weakly singular kernels \cite{LST:1996}.
Jin et al. \cite{JLZ:2017} developed correction BDF$k$ ($k\leq 6$) formulas   to restore the desired $k$-order convergence rate  for subdiffusion \eqref{fee},
which also hold   for the fractional Feynman-Kac equation with L\'{e}vy flight \cite{SC:2020}.
For $L_k$ approximation, Jin et al. \cite{JLZ:2016} revisit the error analysis of $L_1$ scheme, and establish an $\mathcal{O}(\tau)$ convergence rate for both smooth and nonsmooth initial data.
Yan et al.  introduced a modified $L_1$ scheme for solving \eqref{fee} and obtain optimal convergence rate with $\mathcal{O}(\tau^{2-\alpha})$ for smooth and \cite{YKF:2018}.
Wang and Yan et al. proved that   correction schemes of $L_2$  and $L_3$ approximations, respectively,   have the optimal convergence orders $\mathcal{O}(\tau^{3-\alpha})$ and $\mathcal{O}(\tau^{4-\alpha})$ for both smooth and nonsmooth data \cite{WYY:2020}. It seems that there are no published works of high-order  $L_k$ ($k\geq 4$) approximation with  nonsmooth initial data   for subdiffusion  \eqref{fee}.
In fact, it is not an easy task for convergence analysis  based on the idea of \cite{JLZ:2016,WYY:2020,YKF:2018}, this is due to the complexity of the coefficients in the $L_k$ approximation with Bose-Einstein integral. In this work, the key step of designing  correction $L_k$ schemes   is to calculate the explicit form of the coefficients of  $L_k$ approximation with the polylogarithm function or Bose-Einstein integral \cite{BBR:2003}.
Moreover, we need to construct the high-order  $\tau_8$-approximation of order $8$ for  Bose-Einstein integral \cite{BBR:2003}.
Then the desired  $(k+1-\alpha)$th-order convergence rate can be proved
for the correction $L_k$ scheme with nonsmooth data.
The main advantage of $L_k$ approximation (compared with BDF$k$)  is that its convergence rate  higher  than  BDF$k$ method and more easily implemented on variable  grids/graded meshes.

The paper is organized as follows. In Section \ref{Se:corre}, we provide correction of $L_k$ approximation at the starting $k$ steps for fractional order evolution equation \eqref{fee}.
In Section \ref{Se:conver}, we provide the detailed convergence analysis of correction $L_{k}$ schemes.
Some numerical examples are given to show the effectiveness of the presented schemes in Section \ref{Se:numer}.

\section{Correction of  high-order $L_k$ approximation }\label{Se:corre}
Let $V(t)=u(t)-u(0)=u(t)-v$, we can rewrite \eqref{fee} as \cite{Podlubny:1999}
\begin{equation}\label{rfee}
\left \{
\begin{split}
     & ^C_0{D}^{\alpha}_{t} V(t) - A V(t)= Av + f(t), \quad 0<t<T, \\
     & V(0) = 0,
\end{split}
\right.
\end{equation}
where we use   $^C_0{D}^{\alpha}_{t} u(t)=\partial^{\alpha}_{t}(u(t)-u(0))=\partial^{\alpha}_{t}V(t)= \,^C_0{D}^{\alpha}_{t} V(t)$ with $V(0) = 0$, and $\partial^{\alpha}_{t}$ denotes the left-sided Riemann-Liouville fractional derivative of order $\alpha \in (0,1)$
\begin{equation*}\label{RLfd}
  \partial^{\alpha}_{t}u(t) = \frac{1}{\Gamma(1-\alpha)}\frac{\partial}{\partial t} \int^{t}_{0} {(t-s)^{-\alpha}u(s)} ds.
\end{equation*}

\subsection{Derivation of the high-order $L_k$ approximation}
Let $t_{n} = n\tau, n = 0, 1, \ldots, N,$ be a uniform partition of the time interval $[0, T]$ with the step size $\tau = \frac{T}{N}$, and let $V^{n}$ denote the approximation of $V(t_{n})$ and $f^{n} = f(t_{n})$.
Using $k+1$ points $(t_{j-k}$, $V^{j-k})$, $\dots$,   $(t_{j-1}$, $V^{j-1})$,   $(t_{j}, V^{j})$ for $j\geq k$,
we can construct the Lagrange interpolation function $L_{k,j}(t)$ of degree $k$ with $k\leq6$, namely,
\begin{equation}\label{ads2.11}
L_{k,j}[V(t)]=\sum^{k}_{l=0} V^{j-l} \prod^{k}_{i=0,i\neq l} \frac{t-t_{j-i}}{t_{j-l}-t_{j-i}},~~t\in (t_{j-1},t_{j}),
\end{equation}
and its derivative is
\begin{equation}\label{ads2.12}
\begin{split}
L'_{1,j}[V(t)]&=\sum^{k}_{l=0} V^{j-l}   \sum^{k}_{m=0,m\neq l} \frac{1}{t_{j-l}-t_{j-m}},~~k=1;\\
L'_{k,j}[V(t)]&=\sum^{k}_{l=0} V^{j-l}   \sum^{k}_{m=0,m\neq l} \left(\frac{1}{t_{j-l}-t_{j-m}}  \prod^{k}_{i=0,i\neq l,i\neq m} \frac{t-t_{j-i}}{t_{j-l}-t_{j-i}}\right),~~2\leq k \leq6.
\end{split}
\end{equation}

Moreover, we take  $V^{-1}=V^{-2}=\dots=V^{-k}=0$ and $t_{-1}=-t_{1}, t_{-2}=-t_{2}, \dots, t_{-k}=-t_{k}$ such that all quantities appearing in \eqref{ads2.11} and \eqref{ads2.12} are defined for $1\leq j \leq k-1$.
Then the $L_k$ approximation of the Caputo fractional derivative is
\begin{equation}\label{add2.5}
\begin{split}
\delta^{\alpha}_{\tau}V^n
& =\frac{1}{\Gamma(1-\alpha)}\sum^{n}_{j=1} \int_{t_{j-1}}^{t_{j}} (t_{n}-s)^{-\alpha} L'_{k,j}[V(s)] ds  =\sum^{n}_{j=1} \sum^{k}_{l=0} V^{j-l} J(k,j,l)\\
& =\sum^{k}_{l=0}\sum^{n-l}_{j=1}  V^{j} J(k,j+l,l)  = \sum^{n}_{j=1}  \sum^{\min\{n,j+k\}}_{l=j}J(k,l,l-j)V^{j}
=:\tau^{-\alpha} \sum^{n}_{j=1} \omega^{(k)}_{n-j} V^{j}
\end{split}
\end{equation}
with
\begin{equation*}
\begin{split}
J&(k,j,l)=\frac{1}{\Gamma(1-\alpha)}\int_{t_{j-1}}^{t_{j}} (t_{n}-s)^{-\alpha}\sum^{k}_{m=0,m\neq l} \frac{1}{t_{j-l}-t_{j-m}}\prod^{k}_{i=0,i\neq l,i\neq m} \frac{s-t_{j-i}}{t_{j-l}-t_{j-i}} ds\\
=&\frac{1}{\tau^{\alpha}}\frac{1}{\Gamma(1-\alpha)}\int_{0}^{1} (n-j+1-s)^{-\alpha}\sum^{k}_{m=0,m\neq l} \frac{1}{m-l}\prod^{k}_{i=0,i\neq l,i\neq m} \frac{s+i-1}{i-l} ds,~0\leq l\leq k\leq 6.
\end{split}
\end{equation*}
Here the coefficients  $\omega^{(k)}_{n-j}$  are defined by
\begin{equation*}
\begin{split}
\omega^{(1)}_{j}=& \frac{\rho_{1}(j,l,1)}{\Gamma(2-\alpha)}, \qquad \omega^{(2)}_{j}= \frac{\rho_{2}(j,l,2)}{\Gamma(3-\alpha)} + \frac{1}{2}\frac{\rho_{2}(j,l,1)}{\Gamma(2-\alpha)}, \\
\omega^{(3)}_{j}=& \frac{\rho_{3}(j,l,3)}{\Gamma(4-\alpha)} + \frac{\rho_{3}(j,l,2)}{\Gamma(3-\alpha)} + \frac{1}{3}\frac{\rho_{3}(j,l,1)}{\Gamma(2-\alpha)}, \\
\omega^{(4)}_{j}=& \frac{\rho_{4}(j,l,4)}{\Gamma(5-\alpha)} + \frac{3}{2}\frac{\rho_{4}(j,l,3)}{\Gamma(4-\alpha)}+\frac{11}{12}\frac{\rho_{4}(j,l,2)}{\Gamma(3-\alpha)}
                     + \frac{1}{4}\frac{\rho_{4}(j,l,1)}{\Gamma(2-\alpha)}, \\
\omega^{(5)}_{j}=& \frac{\rho_{5}(j,l,5)}{\Gamma(6-\alpha)} + 2\frac{\rho_{5}(j,l,4)}{\Gamma(5-\alpha)} + \frac{7}{4}\frac{\rho_{5}(j,l,3)}{\Gamma(4-\alpha)} + \frac{5}{6}\frac{\rho_{5}(j,l,2)}{\Gamma(3-\alpha)}+ \frac{1}{5}\frac{\rho_{5}(j,l,1)}{\Gamma(2-\alpha)}, \\
\omega^{(6)}_{j}=
& \frac{\rho_{6}(j,l,6)}{\Gamma(7-\alpha)} + \frac{5}{2}\frac{\rho_{6}(j,l,5)}{\Gamma(6-\alpha)} + \frac{17}{6}\frac{\rho_{6}(j,l,4)}{\Gamma(5-\alpha)} + \frac{15}{8}\frac{\rho_{6}(j,l,3)}{\Gamma(4-\alpha)}\\
& + \frac{137}{180}\frac{\rho_{6}(j,l,2)}{\Gamma(3-\alpha)}+ \frac{1}{6}\frac{\rho_{6}(j,l,1)}{\Gamma(2-\alpha)}, ~~ j\geq0
\end{split}
\end{equation*}
with
\begin{equation*}
\rho_{k}(j,l,m) = \sum^{l}_{i=0}  \binom{k+1}{i} (-1)^{i} (j+1-i)^{m-\alpha},~~l=\min\{j+1, k+1\}, ~~k\leq 6.
\end{equation*}
\begin{remark}\label{adr2.1}
It should be noted that  the coefficients $\omega^{(k)}_{j}$ of $L_k$ approximation  in \eqref{add2.5} can be rewritten   the  explicit form with linearly computational count, see  Appendix.
In particularly,  the coefficients $\omega^{(k)}_{j}$ of $L_k$ approximation  reduce to the classical BDF$k$ if $\alpha=1$.
Moreover, the critical angles of $A(\vartheta_k)$-stable are increases when $\alpha$ decreases from $1$ to $0$, see  Table \ref{CALK} and Figures  \ref{fig_subfigures1}-\ref{fig_subfigures}.

\begin{table}[!ht]
\centering
\begin{tabular}{c c c c c c c c c}\hline
    \diagbox{k}{j} & 0& 1 & 2 & 3 & 4 & 5& 6& $\vartheta_k$ \\ \hline
    1 & 1 & -1 &   &   &   &   & &$90^\circ$ \\ \\
    2 & $\frac{3}{2}$ & -2 & $\frac{1}{2}$ &  &  &  & &$90^\circ$ \\ \\
    3 & $\frac{11}{6}$ & -3 & $\frac{3}{2}$ & $-\frac{1}{3}$ &  &  & &$86.03^\circ$ \\ \\
    4 & $\frac{25}{12}$ & -4 & 3 & $-\frac{4}{3}$ & $\frac{1}{4}$ &  & &$73.35^\circ$ \\ \\
    5 & $\frac{137}{60}$ & -5 & 5 & $-\frac{10}{3}$ & $\frac{5}{4}$ & $-\frac{1}{5}$ & &$51.84^\circ$ \\ \\
    6 & $\frac{49}{20}$ & -6 & $\frac{15}{2}$ & $-\frac{20}{3}$ & $\frac{15}{4}$ & $-\frac{6}{5}$ & $\frac{1}{6}$ &$17.84^\circ$ \\
\hline
\end{tabular}
\caption{The coefficients $\omega^{(k)}_{j}$ and critical angles of $A(\vartheta_k)$-stable for  BDF$k$,  see \cite{CD:2015,LeVeque:2007,Lubich:1986}}\label{CALK}
\end{table}
\end{remark}

Then the standard $L_k$ schemes for subdiffusion \eqref{rfee} is as following
\begin{equation}\label{SLks2}
\tau^{-\alpha} \sum^{n}_{j=1} \omega^{(k)}_{n-j} V^{j} - AV^{n} = Av + f(t_{n}), \quad n\geq 1 \quad {\rm with} ~  V^{0}=0,
\end{equation}
where the coefficients  $\omega^{(k)}_{j}$  are given in \eqref{add2.5} or Appendix.

The low regularity of the solution of \eqref{rfee} implies the above  standard $L_k$ approximation \eqref{SLks2}  should  yield a first-order accuracy.
To restore the $(k+1-\alpha)$ order accuracy with nonsmooth data, we correct the standard $L_k$ schemes \eqref{SLks2} at the starting $k$ steps by
\begin{equation}\label{CLks}
\begin{split}
\tau^{-\alpha} \sum^{n}_{j=1} \omega^{(k)}_{n-j} V^{j} - AV^{n} =
& \left(1+ a^{(k)}_{n}\right)(Av+f(0)) \\
&+ \sum_{l=1}^{k-1}\left(\frac{t_{n}^{l}}{l!}+ d^{(k)}_{l,n} \tau^{l} \right)\partial^{l}_{t}f(0)+R_{k}(t_{n}), ~~ 1\leq n \leq k, \\
\tau^{-\alpha} \sum^{n}_{j=1} \omega^{(k)}_{n-j} V^{j} - AV^{n} =& Av + f(0)+\sum_{l=1}^{k-1}\frac{t_{n}^{l}}{l!}\partial^{l}_{t} f(0) +R_{k}(t_{n}),  ~~ k+1 \leq n \leq N.
\end{split}
\end{equation}
Here $R_{k}(t_{n})$ is the corresponding local truncation term, namely,
\begin{equation}\label{srhsR}
R_{k}(t_{n})=f(t_{n}) - f(0) - \sum_{l=1}^{k-1}\frac{t_{n}^{l}}{l!}\partial^{l}_{t} f(0)= \frac{t_{n}^{k}}{k!}\partial^{k}_{t} f(0)
+ \left(\frac{t^{k}}{k!} \ast \partial^{k+1}_{t} f(t)\right)(t_{n}),
\end{equation}
and the symbol $\ast$ denotes Laplace convolution.
The correction coefficients $a^{(k)}_{n}$ and $d^{(k)}_{l,n}$ are given in Table \ref{Ta:corr1} and \ref{Ta:corr2}, respectively.

\begin{table}[!ht]
\begin{center}
\caption{The correction coefficients $a^{(k)}_{n}$.}
\begin{tabular}{c c c c c c c}
\hline
$L_k$  &  $a^{(k)}_{1}$           &  $a^{(k)}_{2}$            &  $a^{(k)}_{3}$           &  $a^{(k)}_{4}$            &  $a^{(k)}_{5}$          &  $a^{(k)}_{6}$            \\ \hline

$k=1$  &  $\frac{1}{2}$           &                           &                          &                           &                         &                           \\

$k=2$  &  $\frac{11}{12}$         &  $-\frac{5}{12}$          &                          &                           &                         &                           \\

$k=3$  &  $\frac{31}{24}$         &  $-\frac{7}{6}$           &  $\frac{3}{8}$           &                           &                         &                           \\

$k=4$  &  $\frac{1181}{720}$      &  $-\frac{177}{80}$        &  $\frac{341}{240}$       &  $-\frac{251}{720}$       &                         &                           \\

$k=5$  &  $\frac{2837}{1440}$     &  $-\frac{2543}{720}$      &  $\frac{17}{5}$          &  $-\frac{1201}{720}$      &  $\frac{95}{288}$       &                           \\

$k=6$  &  $\frac{138241}{60480}$  &  $-\frac{309047}{60480}$  &  $\frac{198251}{30240}$  &  $-\frac{145877}{30240}$  &  $\frac{23077}{12096}$  &  $-\frac{19087}{60480}$   \\ \hline
\end{tabular}
\label{Ta:corr1}
\end{center}
\end{table}

\begin{table}[!ht]
\begin{center}
\caption{The correction coefficients $d^{(k)}_{l,n}$.}
\begin{tabular}{c c c c c c c c}\hline
$L_k$   &         &  $d^{(k)}_{l,1}$        &  $d^{(k)}_{l,2}$         &  $d^{(k)}_{l,3}$       &  $d^{(k)}_{l,4}$        &  $d^{(k)}_{l,5}$      &  $d^{(k)}_{l,6}$    \\ \hline

$k=2$   &  $l=1$  &  $\frac{1}{12}$         &  $0$                     &                        &                         &                       &                     \\\hline

$k=3$   &  $l=1$  &  $\frac{1}{6}$          &  $-\frac{1}{12}$         &  $0$                   &                         &                       &                     \\
        &  $l=2$  &  $0$                    &  $0$                     &  $0$                   &                         &                       &                     \\\hline

$k=4$   &  $l=1$  &  $\frac{59}{240}$       &  $-\frac{29}{120}$       &  $\frac{19}{240}$      &  $0$                    &                       &                     \\
        &  $l=2$  &  $\frac{1}{240}$        &  $-\frac{1}{240}$        &  $0$                   &  $0$                    &                       &                     \\
        &  $l=3$  &  $-\frac{1}{720}$       &  $0$                     &  $0$                   &  $0$                    &                       &                     \\\hline

$k=5$   &  $l=1$  &  $\frac{77}{240}$       &  $-\frac{7}{15}$         &  $\frac{73}{240}$      &  $-\frac{3}{40}$        &  $0$                  &                     \\
        &  $l=2$  &  $\frac{1}{96}$         &  $-\frac{1}{60}$         &  $\frac{1}{160}$       &  $0$                    &  $0$                  &                     \\
        &  $l=3$  &  $-\frac{1}{360}$       &  $\frac{1}{720}$         &  $0$                   &  $0$                    &  $0$                  &                     \\
        &  $l=4$  &  $0$                    &  $0$                     &  $0$                   &  $0$                    &  $0$                  &                     \\\hline

$k=6$   &  $l=1$  &  $\frac{23719}{60480}$  &  $-\frac{11371}{15120}$  &  $\frac{7381}{10080}$  &  $-\frac{5449}{15120}$  &  $\frac{863}{12096}$  &  $0$                \\
        &  $l=2$  &  $\frac{1}{60}$         &  $-\frac{17}{480}$       &  $\frac{1}{40}$        &  $-\frac{1}{160}$       &  $0$                  &  $0$                \\
        &  $l=3$  &  $-\frac{58}{15120}$    &  $\frac{53}{15120}$      &  $-\frac{1}{945}$      &  $0$                    &  $0$                  &  $0$                \\
        &  $l=4$  &  $-\frac{1}{6048}$      &  $\frac{1}{6048}$        &  $0$                   &  $0$                    &  $0$                  &  $0$                \\
        &  $l=5$  &  $\frac{1}{30240}$      &  $0$                     &  $0$                   &  $0$                    &  $0$                  &  $0$                \\ \hline
\end{tabular}
\label{Ta:corr2}
\end{center}
\end{table}

\subsection{Solution representation for \eqref{fee}}

According to \eqref{srhsR}, we can rewrite \eqref{rfee} as
\begin{equation}\label{rfee1}
^C_0{D}^{\alpha}_{t} V(t) - A V(t)= Av + f(0)+\sum_{l=1}^{k-1}\frac{t^{l}}{l!}\partial^{l}_{t} f(0) +R_{k}(t)~~{\rm with} \quad V(0)=0.
\end{equation}
Taking the Laplace transform in both sides of \eqref{rfee1}, it leads to
\begin{equation*}
\widehat{V}(z)= (z^{\alpha}-A)^{-1}\left( z^{-1}\left(Av + f(0)\right)+\sum_{l=1}^{k-1}\frac{1}{z^{l+1}}\partial^{l}_{t} f(0) +\widehat{R_{k}}(z) \right).
\end{equation*}
By the inverse Laplace transform, there exists  \cite{JLZ:2017}
\begin{equation}\label{LT}
\begin{split}
   V(t) &=  \frac{1}{2\pi i} \int_{\Gamma_{\theta, \kappa}} e^{zt}  K(z) \left( Av+f(0)+ z\sum^{k-1}_{l=1} \frac{1}{z^{l+1}} \partial^{l}_{t}f(0) +z\widehat{R_{k}}(z) \right)dz,
\end{split}
\end{equation}
where
\begin{equation}\label{Gamma}
\Gamma_{\theta, \kappa}=\{ z\in \mathbb{C}: |z|=\kappa, |\arg z|\leq \theta \} \cup \{ z\in \mathbb{C}: z=re^{\pm i\theta}, r\geq \kappa \}
\end{equation}
with $\theta \in (\pi/2, \pi)$, $\kappa>0$, and
\begin{equation}\label{K}
K(z)=z^{-1}(z^{\alpha}-A)^{-1}.
\end{equation}
From \cite{LST:1996} and \cite{Thomee:2006}, we know that the operator $A$ satisfies the following resolvent estimate
\begin{equation*}\label{resolvent estimate}
\left\| (z-A)^{-1} \right\| \leq c_{\phi} |z|^{-1} \quad \forall z\in \Sigma_{\phi}
\end{equation*}
for all $\phi\in (\pi/2, \pi)$, where $\Sigma_{\theta}:=\{ z\in \mathbb{C}\backslash \{0\}:|\arg z| < \theta \}$ is a sector of the complex plane $\mathbb{C}$.
Hence, $z^{\alpha}\in \Sigma_{\theta '}$ with $\theta '=\alpha\theta<\theta<\pi$ for all $z\in \Sigma_{\theta}$. Therefore, there exist a positive constant $c$ such that
\begin{equation}\label{fractional resolvent estimate}
\left\| \left(z^{\alpha}-A\right)^{-1} \right\| \leq c |z|^{-\alpha} \quad \forall z\in \Sigma_{\theta}.
\end{equation}

\subsection{Discrete solution representation  for  correction   $L_k$ approximation \eqref{CLks}}
We next provide the following discrete solution for the  subdiffusion \eqref{CLks}.
\begin{lemma}\label{Lemma2.1}
Let $f\in C^{k}([0,T];L^{2}(\Omega))$ and $\int^{t}_{0}{(t-s)^{\alpha-1}}||\partial^{k+1}_{s}f(s)||_{L^{2}(\Omega)}ds<\infty$.  Then
\begin{equation*}
\begin{split}
V^{n}
&=\frac{1}{2\pi i}\int_{\Gamma^{\tau}_{\theta,\kappa}} e^{zt} K(z_{\tau}) \bar{\mu}(e^{-z\tau})(Av+f(0))dz+\frac{1}{2\pi i}\int_{\Gamma^{\tau}_{\theta,\kappa}}e^{zt}z_{\tau}K(z_{\tau})\tau\widetilde{R_{k}}(e^{-z\tau}) dz\\
&\quad+\frac{1}{2\pi i}\int_{\Gamma^{\tau}_{\theta,\kappa}}e^{zt}z_{\tau}K(z_{\tau})\sum^{k-1}_{l=1}\left(\frac{\gamma_{l}(e^{-z\tau})}{l!}+\sum_{j=1}^{k-1}d^{(k)}_{l,j}e^{-zt_{j}}\right)\tau^{l+1}
\partial^{l}_{t}f(0)dz
\end{split}
\end{equation*}
with $\Gamma^{\tau}_{\theta, \kappa}=\{z\in \Gamma_{\theta, \kappa}: |\Im z|\leq \pi / \tau\}$.
Here the coefficients  $d^{(k)}_{l,j}$ are given in Table \ref{Ta:corr2} and
\begin{equation}\label{DLTcoeff}
\begin{array}{rl     rl}
K(z_{\tau})    \!\!\!\!   &= z_{\tau}^{-1}(z_{\tau}^{\alpha}-A)^{-1},   & \qquad   \bar{\mu}(\xi)  \!\!\!\! & = \tau z_{\tau}\left( \frac{\xi}{1-\xi} + \sum_{j=1}^{k}a^{(k)}_{j}\xi^{j}\right), \\
\gamma_{l}(\xi)  \!\!\!\! &= \sum_{n=1}^{\infty}n^{l}\xi^{n}=\left(\xi\frac{d}{d\xi}\right)^{l}\frac{1}{1-\xi},  &~~ \qquad ~ \widetilde{R_{k}}(\xi) \!\!\!\! & = \sum_{n=1}^{\infty} R_{k}(t_{n})\xi^{n},
\end{array}
\end{equation}
and
\begin{equation}\label{SC}
z_{\tau}:=\frac{\delta(\xi)}{\tau}, \quad \delta^{\alpha}(\xi) := \sum_{j=0}^{\infty} \omega^{(k)}_{j} \xi^{j}, \quad \xi=e^{-z\tau},
\end{equation}
where $\omega^{(k)}_{j}$ be given in Appendix and $a^{(k)}_{j}$ be given in Table \ref{Ta:corr1}.
\end{lemma}
\begin{proof}

Multiplying the \eqref{CLks} by $\xi^{n}$ and summing over $n$, we obtain
\begin{equation*}
\begin{split}
&\sum^{\infty}_{n=1} \left(  \tau^{-\alpha} \sum^{n}_{j=1} \omega^{(k)}_{n-j} V^{j} \right) \xi^{n} - \sum^{\infty}_{n=1} \left( AV^{n} \right) \xi^{n}  \\
=& \left(\frac{\xi }{1-\xi} + \sum^{k}_{j=1}a^{(k)}_{j} \xi^{j}\right)(Av+f(0))
+\sum^{k-1}_{l=1}\left(\frac{\gamma_{l}(\xi)}{l!} +\sum^{k-1}_{j=1}d^{(k)}_{l,j} \xi^{j}\right)\tau^{l}\partial^{l}_{t}f(0) + \widetilde{{R}_{k}}(\xi),
\end{split}
\end{equation*}
where $\widetilde{{R}_{k}}(\xi)= \sum^{\infty}_{n=1} R_{k}(t_{n}) \xi^{n}$ and $\gamma_{l}(\xi)=\sum^{\infty}_{n=1}n^{l} \xi^{n}$. Using the equality
\begin{equation*}
\sum^{\infty}_{n=1} \left( \sum^{n}_{j=1} \omega^{(k)}_{n-j} V^{j} \right) \xi^{n}
= \sum^{\infty}_{j=0}\omega^{(k)}_{j} \xi^{j}\sum^{\infty}_{n=1} V^{n}\xi^{n} := \delta^{\alpha}(\xi) \widetilde{V}(\xi),
\end{equation*}
it yields
\begin{equation}\label{ads2.17}
\begin{split}
\widetilde{V}(\xi) = &  \left( \tau^{-\alpha} \delta^{\alpha}(\xi) -A\right)^{-1} \left[\left(\frac{\xi }{1-\xi} + \sum^{k}_{j=1}a^{(k)}_{j} \xi^{j}\right)(Av+f(0)) \right. \\
                 & \qquad \qquad \qquad \qquad \quad \left.+\sum^{k-1}_{l=1}\left(\frac{\gamma_{l}(\xi)}{l!} +\sum^{k-1}_{j=1}d^{(k)}_{l,j} \xi^{j}\right)\tau^{l}\partial^{l}_{t}f(0) + \widetilde{{R}_{k}}(\xi) \right].
\end{split}
\end{equation}
According to  Cauchy's integral formula, and the change of variables $\xi=e^{-z\tau}$, and  Cauchy's theorem, one has
\begin{equation}\label{DLT}
\begin{split}
V^{n}
&=\frac{1}{2\pi i}\int_{\Gamma^{\tau}_{\theta,\kappa}} e^{zt} K(z_{\tau}) \bar{\mu}(e^{-z\tau})(Av+f(0))dz+\frac{1}{2\pi i}\int_{\Gamma^{\tau}_{\theta,\kappa}}e^{zt}z_{\tau}K(z_{\tau})\tau\widetilde{R_{k}}(e^{-z\tau}) dz\\
&\quad+\frac{1}{2\pi i}\int_{\Gamma^{\tau}_{\theta,\kappa}}e^{zt}z_{\tau}K(z_{\tau})\sum^{k-1}_{l=1}\left(\frac{\gamma_{l}(e^{-z\tau})}{l!}
+\sum_{j=1}^{k-1}d^{(k)}_{l,j}e^{-jz\tau}\right)\tau^{l+1}
\partial^{l}_{t}f(0)dz,
\end{split}
\end{equation}
where $\Gamma^{\tau}_{\theta, \kappa}=\{z\in \Gamma_{\theta, \kappa}: |\Im z|\leq \pi / \tau\}$, $d^{(k)}_{l,j}$ be given in Table \ref{Ta:corr2} and
\begin{equation*}
\begin{array}{rl     rl}
K(z_{\tau})    \!\!\!\!   &= z_{\tau}^{-1}(z_{\tau}^{\alpha}-A)^{-1}   & \qquad   \bar{\mu}(\xi)  \!\!\!\! & = \tau z_{\tau}\left( \frac{\xi}{1-\xi} + \sum_{j=1}^{k}a^{(k)}_{j}\xi^{j}\right) \\
\gamma_{l}(\xi)  \!\!\!\! &= \sum_{n=1}^{\infty}n^{l}\xi^{n}=\left(\xi\frac{d}{d\xi}\right)^{l}\frac{1}{1-\xi}  & \qquad  \widetilde{R_{k}}(\xi) \!\!\!\! & = \sum_{n=1}^{\infty} R_{k}(t_{n})\xi^{n}.
\end{array}
\end{equation*}
Here
\begin{equation*}
z_{\tau}:=\frac{\delta(\xi)}{\tau}, \quad \delta^{\alpha}(\xi) := \sum_{j=0}^{\infty} \omega^{(k)}_{j} \xi^{j}, \quad \xi=e^{-z\tau},
\end{equation*}
and the coefficients  $\omega^{(k)}_{j}$ are  given in  Appendix and $a^{(k)}_{j}$ are defined by Table \ref{Ta:corr1}.
The proof is completed.
\end{proof}

\section{Convergence analysis}\label{Se:conver}
In this section, we provide the detailed convergence analysis of correction $L_{k}$  approximation for the subdiffusion \eqref{fee}.
First, we give some lemmas that
will be used.

\subsection{A few technical lemmas}
We introduce the polylogarithm function or Bose-Einstein integral \cite{BBR:2003} as following
\begin{equation}\label{polylogarithm function}
Li_{p}(\xi)= \sum_{j=1}^{\infty} \frac{\xi^{j}}{j^{p}}.
\end{equation}

\begin{lemma}{\cite{JLZ:2016}}\label{LipSE}
Let $p\neq 1, 2, \ldots$, the polylogarithm function $Li_{p}(e^{-z})$ satisfies the following singular expansion
\begin{equation*}\label{LpSE}
Li_{p}(e^{-z}) \sim \Gamma(1-p)z^{p-1} + \sum_{j=0}^{\infty} (-1)^{j} \zeta(p-j)\frac{z^{j}}{j!} \quad {\rm as}~ z\rightarrow 0,
\end{equation*}
where $\zeta$ denotes the Riemann zeta function, namely, $\zeta(p)=Li_{p}(1)$.
\end{lemma}

\begin{lemma}\label{Lemma:LipCA}
Let $|z\tau|\leq \frac{\pi}{\sin\theta}$ and  $\theta>\pi/2$ be close to $\pi/2$, and $p=\alpha-k$ with $\alpha\in(0,1), k\leq 6$. The series
\begin{equation}\label{LpSE}
Li_{p}(e^{-z\tau}) = \Gamma(1-p)(z\tau)^{p-1} + \sum_{j=0}^{\infty} (-1)^{j} \zeta(p-j)\frac{(z\tau)^{j}}{j!}
\end{equation}
converges absolutely.
\end{lemma}

\begin{proof}
The similar arguments can be performed as Lemma 3.4 in \cite{JLZ:2016}, we
omit it here.
\end{proof}

\begin{lemma}\label{Lemma:WightSE}
Let $\omega^{(k)}_{j}$ be given in  Appendix and $\xi=e^{-z\tau}$.
Then the following singularity expansion holds
\begin{equation*}\label{SE}
\sum_{j=0}^{\infty} \omega^{(k)}_{j} \xi^{j}= (z\tau)^{\alpha} + c^{(k)}_{k+1}(z\tau)^{k+1} + c^{(k)}_{k+2}(z\tau)^{k+1+\alpha} + \ldots
\end{equation*}
for some suitable constants $c^{(k)}_{k+1}, c^{(k)}_{k+2}, \ldots$.
\end{lemma}
\begin{proof}
From  (2.13) of \cite{WYY:2020} and  the coefficients $\omega^{(k)}_{j}$ in Appendix, it is easy to check
\begin{equation}\label{ad3.5}
\begin{split}
 \sum_{j=0}^{\infty} \omega^{(k)}_{j} \xi^{j}
& =\frac{(1-\xi)^{k+1}}{\xi} \sum_{j=1}^{k} \frac{ b^{(k)}_{k+1-j} Li_{\alpha-j}(\xi)}{\Gamma(j+1-\alpha)}\\
& =\frac{(1-e^{-z\tau})^{k+1}}{e^{-z\tau}} \sum_{j=1}^{k} \frac{ b^{(k)}_{k+1-j} Li_{\alpha-j}(e^{-z\tau})}{\Gamma(j+1-\alpha)}.
 \end{split}
\end{equation}
Here the coefficients  $b^{(k)}_{j}$ are  given in Table \ref{Ta:corr5}
and $Li_{p}(z)$ denotes the polylogarithm function  with $p=\alpha-j$ in Lemma \ref{Lemma:LipCA}.

Using  Taylor expansion and Lemma \ref{Lemma:LipCA}, for some suitable constants $d_0, \ldots$, we obtain
\begin{equation*}
\frac{(1-e^{-z\tau})^{k+1}}{e^{-z\tau}}= \sum_{j=1}^{k+1}  \bar{b}^{(k)}_{j} (z\tau)^{k+j} + \mathcal{O}\left((z\tau)^{2k+2}\right),
\end{equation*}
and
\begin{equation*}
\sum_{j=1}^{k} \frac{ b^{(k)}_{k+1-j} Li_{\alpha-j}(e^{-z\tau})}{\Gamma(j+1-\alpha)}=\sum_{j=1}^{k}  b^{(k)}_{k+1-j} (z\tau)^{\alpha-1-j}+ d_{0}(z\tau)^{0} + \ldots
\end{equation*}
with $b^{(k)}_{j}, \bar{b}^{(k)}_{j}$ in Table \ref{Ta:corr5}.
\begin{table}[!ht]
\begin{center}
\caption{The coefficients $b^{(k)}_{j}$ and $\bar{b}^{(k)}_{j}$.}
\begin{tabular}{c c c c c c c|c c c c c c c }
\hline
$L_k$  &  $b^{(k)}_{1}$       &  $b^{(k)}_{2}$       &  $b^{(k)}_{3}$       &  $b^{(k)}_{4}$        &  $b^{(k)}_{5}$          &  $b^{(k)}_{6}$
       &  $\bar{b}^{(k)}_{1}$ &  $\bar{b}^{(k)}_{2}$ &  $\bar{b}^{(k)}_{3}$ &  $\bar{b}^{(k)}_{4}$  &  $\bar{b}^{(k)}_{5}$    &  $\bar{b}^{(k)}_{6}$  &  $\bar{b}^{(k)}_{7}$     \\ \hline

$k=1$  &  $1$                 &                      &                      &                       &                         &
       &  $1$                 &  $0$                 &                      &                       &                         &                       &                          \\

$k=2$  &  $1$                 &  $\frac{1}{2}$       &                      &                       &                         &
       &  $1$                 &  $-\frac{1}{2}$      &  $\frac{1}{4}$       &                       &                         &                       &                          \\

$k=3$  &  $1$                 &  $1$                 &  $\frac{1}{3}$       &                       &                         &
       &  $1$                 &  $-1$                &  $\frac{2}{3}$       &  $-\frac{1}{3}$       &                         &                       &                          \\

$k=4$  &  $1$                 &  $\frac{3}{2}$       &  $\frac{11}{12}$     &  $\frac{1}{4}$        &                         &
       &  $1$                 &  $-\frac{3}{2}$      &  $\frac{4}{3}$       &  $-\frac{7}{8}$       &  $\frac{67}{144}$       &                       &                          \\

$k=5$  &  $1$                 &  $2$                 &  $\frac{7}{4}$       &  $\frac{5}{6}$        &  $\frac{1}{5}$          &
       &  $1$                 &  $-2$                &  $\frac{9}{4}$       &  $-\frac{11}{6}$      &  $\frac{287}{240}$      &  $-\frac{79}{120}$    &                          \\

$k=6$  &  $1$                 &  $\frac{5}{2}$       &  $\frac{17}{6}$      &  $\frac{15}{8}$       &  $\frac{137}{180}$      &  $\frac{1}{6}$
       &  $1$                 &  $-\frac{5}{2}$      &  $\frac{41}{12}$     &  $-\frac{100}{3}$     &  $\frac{619}{240}$      &  $-\frac{241}{144}$   &  $\frac{509}{540}$       \\ \hline
\end{tabular}
\label{Ta:corr5}
\end{center}
\end{table}

According to the above equations, it yields
\begin{equation*}\label{Lk}
\begin{split}
\sum_{j=0}^{\infty} \omega_{j} \xi^{j}
=& \left( \sum_{j=1}^{k+1}  \bar{b}^{(k)}_{j} (z\tau)^{k+j} + \mathcal{O}((z\tau)^{2k+2})  \right) \left(\sum_{j=1}^{k}  b^{(k)}_{k+1-j} (z\tau)^{\alpha-1-j} + d_{0}(z\tau)^{0} + \ldots \right) \\
=& (z\tau)^{\alpha}  + c^{(k)}_{k+1}(z\tau)^{k+1}+ c^{(k)}_{k+2}(z\tau)^{k+1+\alpha} +  \ldots,
\end{split}
\end{equation*}
for some suitable constants $c^{(k)}_{k+1}, c^{(k)}_{k+2}, \ldots$.
The proof is completed.
\end{proof}

\begin{lemma}\label{Lemma1}
Let $z_{\tau}$ be given by \eqref{SC} with $1\leq k \leq 6$. Then for all $z\in\Gamma^{\tau}_{\theta, \kappa}$, there exist $c_{1}, c_{2}, c >0$ such that
\begin{equation*}
|z_{\tau}-z| \leq c\tau^{k+1-\alpha}|z|^{k+2-\alpha},~~
c_{1}|z|\leq|z_{\tau}|\leq c_{2}|z| ~~{\rm and} ~~ |z^{\alpha}_{\tau} - z^{\alpha}| \leq c \tau^{k+1-\alpha}|z|^{k+1} .
\end{equation*}
\end{lemma}

\begin{proof}
From  \eqref{SC} and Lemma \ref{Lemma:WightSE}, one has
\begin{equation*}
\begin{split}
z_{\tau}-z
     & = \frac{\delta(e^{-z\tau})}{\tau}-z =\frac{\left(\sum_{j=0}^{\infty} \omega^{(k)}_{j} e^{-jz\tau} \right)^{\frac{1}{\alpha}}-z\tau}{\tau}
       = \frac{\left((z\tau)^{\alpha} + c^{(k)}_{k+1}(z\tau)^{k+1} + \ldots\right)^{\frac{1}{\alpha}}-z\tau}{\tau}  \\
     & =\frac{ (z\tau) \left(1 + c^{(k)}_{k+1}(z\tau)^{k+1-\alpha} + \ldots\right)^{\frac{1}{\alpha}}-z\tau}{\tau}
       =\frac{ (z\tau) \left(1 + \frac{c^{(k)}_{k+1}}{\alpha}(z\tau)^{k+1-\alpha} + \ldots\right)-z\tau}{\tau} \\
     & = \mathcal{O}\left(z^{k+2-\alpha}\tau^{k+1-\alpha}\right),
\end{split}
\end{equation*}
\begin{equation*}
\frac{\left|z_{\tau}\right|}{\left|z\right|}  = \frac{ \left|\delta(e^{-z\tau})\right|}{\left|z\tau\right|}  =\left| \left(1 + c^{(k)}_{k+1}(z\tau)^{k+1-\alpha} + \ldots\right)^{\frac{1}{\alpha}} \right|
= 1 + \mathcal{O}\left(z^{k+1-\alpha}\tau^{k+1-\alpha}\right),
\end{equation*}
and
\begin{equation*}
\begin{split}
z^{\alpha}_{\tau} - z^{\alpha}
     & = \frac{\delta^{\alpha}(e^{-z\tau})}{\tau^{\alpha}}-z^{\alpha} =\frac{\sum_{j=0}^{\infty} \omega^{(k)}_{j} e^{-jz\tau} -(z\tau)^{\alpha}}{\tau^{\alpha}}
       = \frac{(z\tau)^{\alpha} + c^{(k)}_{k+1}(z\tau)^{k+1} + \ldots-(z\tau)^{\alpha}}{\tau^{\alpha}}  \\
     & = \mathcal{O}\left(z^{k+1}\tau^{k+1-\alpha}\right).
\end{split}
\end{equation*}
The proof is completed.
\end{proof}

\begin{lemma}\label{Lemma2}
Let $\xi=e^{-z\tau}$ and $z\in\Gamma^{\tau}_{\theta, \kappa}$.
Let $\bar{\mu}(\xi)$ be defined in \eqref{DLTcoeff}. Then
\begin{equation*}
|\bar{\mu}(e^{-z\tau})-1|\leq C (|z|\tau)^{k+1-\alpha}.
\end{equation*}
\end{lemma}

\begin{proof}
From  \eqref{DLTcoeff} and \eqref{SC}, there exists
\begin{equation*}
\begin{split}
\bar{\mu}(e^{-z\tau})
 &  =  \tau z_{\tau}\left( \frac{e^{-z\tau}}{1-e^{-z\tau}} + \sum_{j=1}^{k}a^{(k)}_{j}e^{-jz\tau}\right) = \delta(e^{-z\tau})\left( \frac{e^{-z\tau}}{1-e^{-z\tau}} + \sum_{j=1}^{k}a^{(k)}_{j}e^{-jz\tau}\right) \\
 &  =  \delta(e^{-z\tau})\left( \frac{e^{-z\tau}}{1-e^{-z\tau}}+\sum_{j=1}^{k}a^{(k)}_{j}e^{-jz\tau}\right)=\frac{\delta(e^{-z\tau})}{1-e^{-z\tau}}\left(e^{-z\tau}+\sum_{j=1}^{k}a^{(k)}_{j}e^{-jz\tau}(1-e^{-z\tau})\right)\\
 &  =   \left(1 + \mathcal{O}((z\tau)^{k+1-\alpha}) \right) \frac{z\tau}{1-e^{-z\tau}} \left(e^{-z\tau}+\sum_{j=1}^{k}a^{(k)}_{j}e^{-jz\tau}(1-e^{-z\tau})\right),
\end{split}
\end{equation*}
where the coefficients $a^{(k)}_{j}$ are given in Table \ref{Ta:corr1}.
Using  Taylor series expansion for  $e^{-z\tau}$, $e^{-jz\tau}(1-e^{-z\tau})$ and
\begin{equation*}
 \frac{z\tau}{1-e^{-z\tau}}   =  1 + \frac{1}{2} z\tau +  \frac{1}{12} (z\tau)^2 -  \frac{1}{720} (z\tau)^4  +  \frac{1}{30240} (z\tau)^6 +  \mathcal{O}\left((z\tau)^{8}\right),
\end{equation*}
which is easy to get
\begin{equation*}
|\bar{\mu}(e^{-z\tau})-1|\leq C (|z|\tau)^{k+1-\alpha}.
\end{equation*}
The proof is completed.
\end{proof}

\begin{lemma}\label{Lemma3}
Let $z\in\Gamma^{\tau}_{\theta, \kappa}$ and $z_{\tau}$  be defined in \eqref{SC}. Let $K(z)$  be given in \eqref{K} and $K(z_{\tau})$, $\bar{\mu}$ be given in \eqref{DLTcoeff}. Then
\begin{equation*}
\begin{array}{rlrr}
 \|K(z_{\tau})-K(z)\|\leq &\!\!\!C \tau^{k+1-\alpha} |z|^{k-2\alpha},  &  \|\bar{\mu}(e^{-z\tau})K(z_{\tau})-K(z)\|\leq &\!\!\!C\tau^{k+1-\alpha} |z|^{k-2\alpha},  \\
 \|K(z_{\tau})A-K(z)A\|\leq &\!\!\!C\tau^{k+1-\alpha} |z|^{k-\alpha},  &  \|\bar{\mu}(e^{-z\tau})K(z_{\tau})A-K(z)A\|\leq &\!\!\!C\tau^{k+1-\alpha} |z|^{k-\alpha}.
\end{array}
\end{equation*}
\end{lemma}

\begin{proof}
From \eqref{K}, we have $K(z)A= -z^{-1} + z^{\alpha -1}(z^{\alpha}-A)^{-1}$.
Following the proof of (4.6), (3.12) in   \cite{LST:1996} and noting $\|K'(z)\| \leq C|z|^{-2-\alpha}$, $\|(K(z)A)'\| \leq C|z|^{-2}$, it implies that
\begin{equation*}
  \left\|K(z_{\tau})-K(z)\right\|  \leq  C|z|^{-2-\alpha} |z|^{k+2-\alpha}\tau^{k+1-\alpha} = C \tau^{k+1-\alpha} |z|^{k-2\alpha},
\end{equation*}
and
\begin{equation*}
  \left\|K(z_{\tau})A-K(z)A\right\|  \leq  C|z|^{-2} |z|^{k+2-\alpha}\tau^{k+1-\alpha} = C\tau^{k+1-\alpha} |z|^{k-\alpha}.
\end{equation*}
According to  Lemma \ref{Lemma2} and  $\|K(z)\| \leq C|z|^{-1-\alpha}$, $\|K(z_{\tau})\| \leq C|z|^{-1-\alpha}$, $\|K(z)A\| \leq C|z|^{-1}$, $\|K(z_{\tau})A\| \leq C|z|^{-1}$, we obtain
\begin{equation*}
\begin{split}
\|\bar{\mu}(e^{-z\tau})K(z_{\tau})-K(z)\|
\leq & \left\|\left(\bar{\mu}(e^{-z\tau}) -1\right)K(z_{\tau}) \right\| + \left\| K(z_{\tau})-K(z) \right\| \\
\leq & C (|z|\tau)^{k+1-\alpha} |z|^{-1-\alpha} + C\tau^{k+1-\alpha} |z|^{k-2\alpha} \leq  C\tau^{k+1-\alpha} |z|^{k-2\alpha},
\end{split}
\end{equation*}
and
\begin{equation*}
\begin{split}
\left\|\bar{\mu}(e^{-z\tau})K(z_{\tau})A-K(z)A \right\|
\leq & \left\|\left(\bar{\mu}(e^{-z\tau}) -1\right)K(z_{\tau})A \right\| + \left\| K(z_{\tau})A-K(z)A \right\| \\
\leq & C (|z|\tau)^{k+1-\alpha} |z|^{-1} + C\tau^{k+1-\alpha} |z|^{k-\alpha} \leq  C\tau^{k+1-\alpha} |z|^{k-\alpha}.
\end{split}
\end{equation*}
The proof is completed.
\end{proof}

\begin{lemma}\label{Lemma4}
Let $\xi=e^{-z\tau}$ and $z\in \Gamma^{\tau}_{\theta, \kappa}$. Let $\gamma_{l}(\xi)$ and $d^{(k)}_{l,j}$ be defined in \eqref{DLTcoeff} and Table \ref{Ta:corr2}, respectively. Then
\begin{equation*}
\left|\left( \frac{\gamma_{l}(e^{-z\tau})}{l!}+\sum_{j=1}^{k}d^{(k)}_{l,j}e^{-zj\tau} \right)\tau^{l+1} - \frac{1}{z^{l+1}}  \right| \leq C \tau^{k+1}|z|^{k-l}.
\end{equation*}
\end{lemma}
\begin{proof}
The similar arguments can be performed as in \cite[Lemma 3.2]{SC:2020}, we omit it here.
\end{proof}

\begin{lemma}\label{Lemmaads3.8}
Let $\xi=e^{-z\tau}$ and $z\in \Gamma^{\tau}_{\theta, \kappa}$. Let $\gamma_{l}(\xi)$ be defined in \eqref{DLTcoeff}. Then
\begin{equation*}
\left| \frac{\gamma_{l}(e^{-z\tau})}{l!} \tau^{l+1} - \frac{1}{z^{l+1}}  \right| \leq C \tau^{k+1}.
\end{equation*}
\end{lemma}
\begin{proof}
The similar arguments can be performed as in \cite[Lemma 3.6]{SC:2020}, we omit it here.
\end{proof}

Based on the idea of  \cite{WYY:2020,YKF:2018}, it is hard  to offer a rigorous   proof for  $z^{\alpha}_{\tau}\in \Sigma_{\theta_{0}}$ with $k=4,5,6$,
since the complexity of the  coefficients of  $L_k$ approximation with Bose-Einstein integral in \eqref{ad3.5}.
In a sense the computer has introduced into mathematics the idea of verification of results as happens in the natural sciences \cite{App:1984}.
Then we can check    $z^{\alpha}_{\tau} \in \Sigma_{\theta_{0}}$ with $\alpha \in [0,1)$  using  evaluating the  polylogarithm function \eqref{LipSE} or  Bose-Einstein integral by computer.
For $\alpha=1$, we know that $z^{1}_{\tau}:=z^{\alpha}_{\tau}\in \Sigma_{\theta_{0}}, \forall z\in \Sigma_{\theta}$ by \cite{JLZ:2017} and \cite{SC:2020}.
Moreover, the critical angles of $A(\vartheta_k)$-stable are increases when $\alpha$ decreases from $1$ to $0$, see  Table \ref{CALK} and Figures  \ref{fig_subfigures1}-\ref{fig_subfigures}.

\begin{lemma}\label{Lemma5}
Let $\theta>\pi/2$ be close to $\pi/2$ and $z_{\tau}$ be given by \eqref{SC}. Then there exists $\theta_{0} \in (\pi/2, \pi)$ such that
\begin{equation*}
z^{\alpha}_{\tau}\in \Sigma_{\theta_{0}}  \quad \forall z\in \Sigma_{\theta}.
\end{equation*}
\end{lemma}
\begin{proof}
We next check  $z^{\alpha}_{\tau} \in \Sigma_{\theta_{0}}$ with $\alpha \in [0,1)$  by evaluating the  polylogarithm function \eqref{LipSE} or  Bose-Einstein integral  \cite{BBR:2003}.
More concretely, developed  the idea of  (27) in \cite{BBR:2003},
we  construct the high-order  $\tau_8$-approximation of order $8$ for  \eqref{LpSE}, namely,
\begin{equation}\label{ad3.7}
Li_{p}(\xi)=Li_{p}(e^{-z\tau})\approx\Gamma(1-p)(z\tau)^{p-1}+\frac{\Phi(z\tau)}{\Psi(z\tau)},~~\xi=e^{-z\tau}.
\end{equation}
Here
\begin{equation*}
\begin{split}
\Phi(z\tau)=& b_{0} - z\tau\left(b_{1}-\frac{8b_{0}b_{8}}{15b_{7}}\right) +(z\tau)^{2} \left( \frac{b_{2}}{2} + \frac{b_{0}b_{8}}{15b_{6}} - \frac{8b_{1}b_{8}}{15b_{7}} \right) \\
  & - (z\tau)^{3} \left( \frac{b_{3}}{6} - \frac{4b_{0}b_{8}}{195b_{5}} +  \frac{2b_{1}b_{8}}{15b_{6}} - \frac{4b_{2}b_{8}}{15b_{7}} \right) \\
  & +(z\tau)^{4} \left( \frac{b_{4}}{24} +\frac{b_{0}b_{8}}{312b_{4}}- \frac{4b_{1}b_{8}}{195b_{5}} +  \frac{b_{2}b_{8}}{15b_{6}} - \frac{4b_{3}b_{8}}{45b_{7}} \right) \\
  & -(z\tau)^{5} \left( \frac{b_{5}}{120} - \frac{b_{0}b_{8}}{6435b_{3}} + \frac{b_{1}b_{8}}{312b_{4}}- \frac{2b_{2}b_{8}}{195b_{5}} +  \frac{b_{3}b_{8}}{45b_{6}} - \frac{b_{4}b_{8}}{45b_{7}} \right)\\
  & +(z\tau)^{6} \left( \frac{b_{6}}{720} +\frac{b_{0}b_{8}}{128700b_{2}}- \frac{b_{1}b_{8}}{6435b_{3}}+\frac{b_{2}b_{8}}{624b_{4}}- \frac{2b_{3}b_{8}}{585b_{5}}  +  \frac{b_{4}b_{8}}{180b_{6}} - \frac{b_{5}b_{8}}{225b_{7}} \right)\\
  & -(z\tau)^{7} \left( \frac{b_{7}}{5040} - \frac{b_{0}b_{8}}{4054050b_{1}} + \frac{b_{1}b_{8}}{128700b_{2}}- \frac{b_{2}b_{8}}{12870b_{3}} \right.\\
  & \qquad \qquad \qquad \quad \left.+ \frac{b_{3}b_{8}}{1872b_{4}}- \frac{b_{4}b_{8}}{1170b_{5}} +  \frac{b_{5}b_{8}}{900b_{6}} - \frac{b_{6}b_{8}}{1350b_{7}} \right),
\end{split}
\end{equation*}
and
\begin{equation*}
\begin{split}
\Psi(z\tau)=&1+z\tau\frac{8b_{8}}{15b_{7}}+(z\tau)^{2}\frac{2b_{8}}{15b_{6}} +(z\tau)^{3}\frac{4b_{8}}{195b_{5}} +(z\tau)^{4}\frac{b_{8}}{312b_{4}} +(z\tau)^{5}\frac{b_{8}}{6435b_{3}}\\
& +(z\tau)^{6}\frac{b_{8}}{128700b_{2}} +(z\tau)^{7}\frac{b_{8}}{4054050b_{1}} +(z\tau)^{8}\frac{b_{8}}{259459200b_{0}}
\end{split}
\end{equation*}
 with
 $$b_i=\eta(p-i).$$
Moreover,
\begin{equation*}\label{tau6approximant}
\eta(p)=\frac{2^{p-1}}{2^{p-1}-1} \frac{ 1+36~2^{p} S_{2} + 315~3^{p}S_{3} + 1120~4^{p}S_{4} + 1890~5^{p}S_{5} + 1512~6^{p}S_{6} + 462~7^{p}S_{7}  }{1+36~2^{p}  + 315~3^{p} + 1120~4^{p} + 1890~5^{p} + 1512~6^{p} + 462~7^{p} },
\end{equation*}
and
\begin{equation*}\label{eta function}
S_l=\sum^{l}_{j=1}(-1)^{j} \frac{1}{j^{p}}.
\end{equation*}
From  \eqref{ad3.5}, \eqref{ad3.7}  and the boundary locus method \cite[p. 162]{LeVeque:2007}, there exists
\begin{equation*}
z^{\alpha}_{\tau}=\frac{1}{\tau^\alpha} \sum_{j=0}^{\infty} \omega_{j} \xi^{j}
 =\frac{(1-\xi)^{k+1}}{\xi} \sum_{j=1}^{k} \frac{ b_{k+1-j} Li_{\alpha-j}(\xi)}{\Gamma(j+1-\alpha)} \in \Sigma_{\theta_{0}},
\end{equation*}
since   the critical angles of $A(\vartheta_k)$-stable are increases when $\alpha$ decreases from $1$ to $0$,  see Figures \ref{fig_subfigures1}-\ref{fig_subfigures}; and
 $z^{\alpha}_{\tau}\in \Sigma_{\theta_{0}}, \forall z\in \Sigma_{\theta}$  if $\alpha=1$ in \cite{JLZ:2017,SC:2020}.
Then the desired result is obtained.
\end{proof}

\begin{figure}[!ht]
\centering
\subfigure[hologram of $z^{\alpha}_{\tau}$ for all $\alpha$]
{
      \includegraphics[width=1.5in]{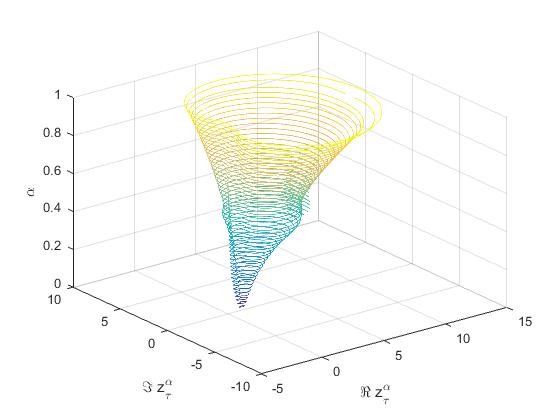}
}
\subfigure[top view of $z^{\alpha}_{\tau}$ for all $\alpha$]
{
      \includegraphics[width=1.5in]{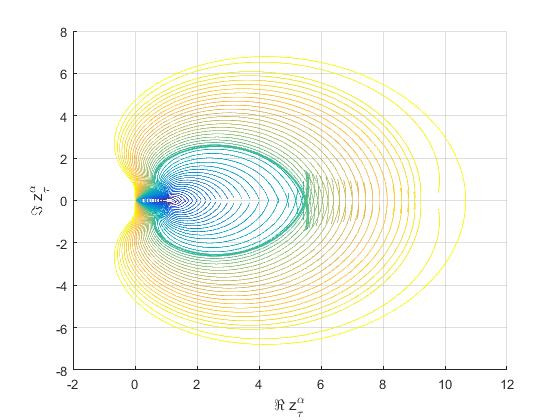}
}
\subfigure[plane graph of $z^{\alpha}_{\tau}$ for some $\alpha$]
{
      \includegraphics[width=1.5in]{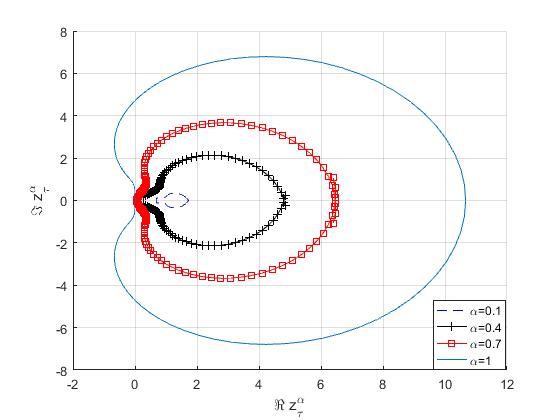}
}
\caption{$z^{\alpha}_{\tau}\in \Sigma_{\theta_{0}}$ for $k=4$.}
\label{fig_subfigures1}
\centering
  \vspace{0.4cm}
\subfigure[hologram of $z^{\alpha}_{\tau}$ for all $\alpha$]
{
      \includegraphics[width=1.5in]{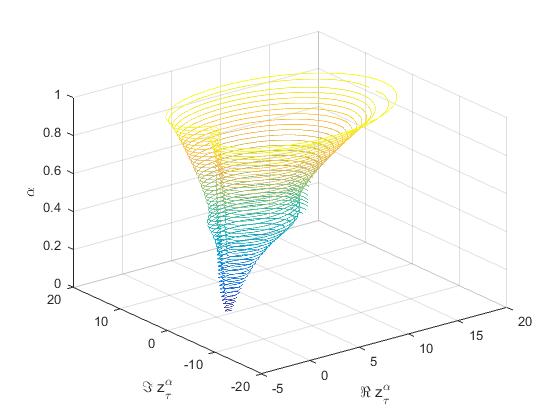}
}
\subfigure[top view of $z^{\alpha}_{\tau}$ for all $\alpha$]
{
      \includegraphics[width=1.5in]{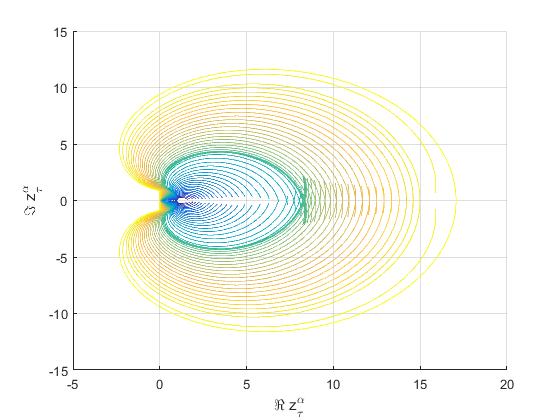}
}
\subfigure[plane graph of $z^{\alpha}_{\tau}$ for some $\alpha$]
{
      \includegraphics[width=1.5in]{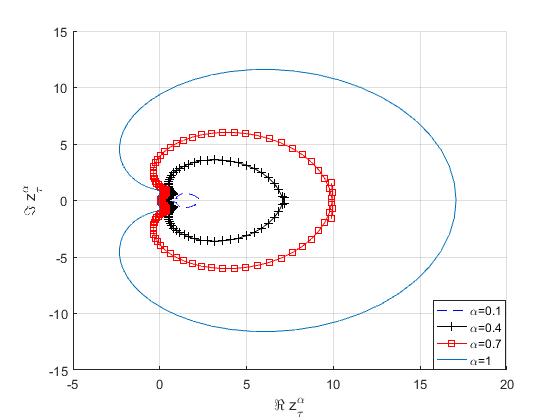}
}
\caption{$z^{\alpha}_{\tau}\in \Sigma_{\theta_{0}}$ for $k=5$.}
\centering
  \vspace{0.4cm}
\subfigure[hologram of $z^{\alpha}_{\tau}$ for all $\alpha$]
{
      \includegraphics[width=1.5in]{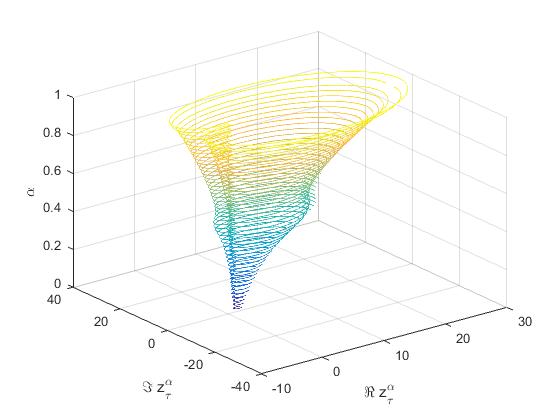}
}
\subfigure[top view of $z^{\alpha}_{\tau}$ for all $\alpha$]
{
      \includegraphics[width=1.5in]{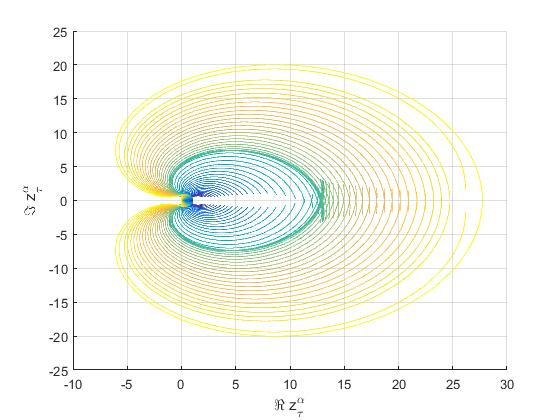}
}
\subfigure[plane graph of $z^{\alpha}_{\tau}$ for some $\alpha$]
{
      \includegraphics[width=1.5in]{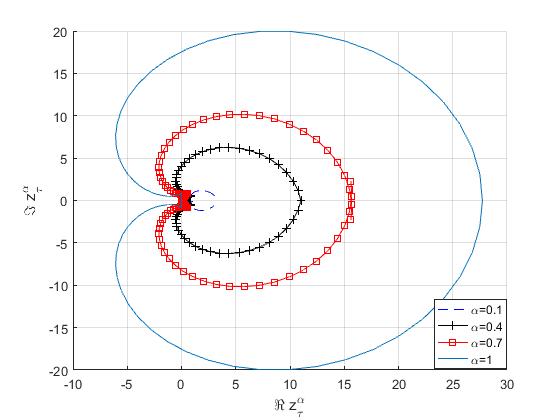}
}
\caption{$z^{\alpha}_{\tau}\in \Sigma_{\theta_{0}}$ for $k=6$.}
\label{fig_subfigures}
\end{figure}

\subsection{Error analysis for subdiffusion}
We now give the error analysis of correction $L_k$ approximation  \eqref{CLks} for \eqref{rfee}. From \eqref{srhsR}, we know that
\begin{equation*}
R_{k}(t_{n})= \frac{t_{n}^{k}}{k!}\partial^{k}_{t} f(0)
+ \left(\frac{t^{k}}{k!} \ast \partial^{k+1}_{t} f(t)\right)(t_{n}).
\end{equation*}
Then  we introduce  the following results.
\begin{lemma}\label{lemma3.9}
Let $V(t_{n})$ and $V^{n}$ be the solutions of \eqref{rfee} and \eqref{CLks}, respectively.
If $v=0$ and $f(t):=\frac{t^{k}}{k!}\partial^{k}_{t} f(0)$, then
\begin{equation*}
\left\| V(t_{n}) - V^{n} \right\|_{L^{2}(\Omega)} \leq c\tau^{k+1}t_{n}^{\alpha-1} \left\| \partial^{k}_{t} f(0) \right\|.
\end{equation*}
\end{lemma}
\begin{proof}
Using    \eqref{LT} and \eqref{DLT}, there exist
\begin{equation*}
V(t_{n})=\frac{1}{2\pi i}\int_{\Gamma_{\theta,\kappa}}{e^{zt_{n}}(z^{\alpha}-A)^{-1}\frac{1}{z^{k+1}}\partial^{k}_{t} f(0)}dz,
\end{equation*}
and
\begin{equation*}
V^{n}=\frac{1}{2\pi i}\int_{\Gamma^{\tau}_{\theta,\kappa}}e^{zt_{n}}\left(z_{\tau}^{\alpha} - A\right)^{-1}
\frac{\gamma_{k}(e^{-z\tau})}{k!}\tau^{k+1} \partial^{k}_{t} f(0) dz,
\end{equation*}
where  $\theta\in (\pi/2,\pi)$ is  sufficiently close to $\pi/2$ and $\kappa=t_{n}^{-1}$ in \eqref{Gamma}.
It leads to
\begin{equation*}
\begin{split}
&V(t_{n})-V^{n}=J_1 + J_2
\end{split}
\end{equation*}
with
\begin{equation*}
\begin{split}
J_1
=\frac{1}{2\pi i}\int_{\Gamma^{\tau}_{\theta,\kappa}}e^{zt_{n}}\left(\left(z^{\alpha} - A\right)^{-1}\frac{1}{z^{k+1}}
-\left(z_{\tau}^{\alpha} - A\right)^{-1}\frac{\gamma_{k}(e^{-z\tau})}{k!}\tau^{k+1}\right) \partial^{k}_{t} f(0) dz,
\end{split}
\end{equation*}
and
\begin{equation*}
\begin{split}
J_2
=\frac{1}{2\pi i}\int_{\Gamma_{\theta,\kappa}\setminus\Gamma^{\tau}_{\theta,\kappa}}{e^{zt_{n}}\left(z^{\alpha} - A\right)^{-1}\frac{1}{z^{k+1}}\partial^{k}_{t} f(0)}dz.
\end{split}
\end{equation*}
According to the triangle inequality, \eqref{fractional resolvent estimate} and Lemmas  \ref{Lemma1},\ref{Lemma3},\ref{Lemmaads3.8},\ref{Lemma5}, one has
\begin{equation*}
\begin{split}
  \| J_1 \|
  &  \leq c\tau^{k+1-\alpha}\|\partial^{k}_{t} f(0)\|\left( \int^{\frac{\pi}{\tau\sin\theta}}_{\kappa} e^{rt_{n}\cos\theta} r^{-2\alpha}dr +\int^{\theta}_{-\theta}e^{\kappa t_{n} \cos\psi} \kappa^{1-2\alpha} d\psi \right)\\
  &  \leq  c\tau^{k+1-\alpha}t_{n}^{2\alpha-1}\|\partial^{k}_{t} f(0)\|,
\end{split}
\end{equation*}
for the last inequality,   we use
\begin{equation}\label{ad3.3.09}
\begin{split}
& \int^{\frac{\pi}{\tau\sin\theta}}_{\kappa} e^{rt_{n}\cos\theta} r^{-2\alpha}dr = t_n^{2\alpha-1} \int^{\frac{t_n\pi}{\tau\sin\theta}}_{t_n\kappa} e^{s\cos\theta} s^{-2\alpha}ds  \leq c t_n^{2\alpha-1},\\
& \int^{\theta}_{-\theta}e^{\kappa t_{n} \cos\psi} \kappa^{1-2\alpha} d\psi
= t_{n}^{2\alpha-1}\int^{\theta}_{-\theta}e^{\kappa t_{n} \cos\psi} \left(\kappa t_{n}\right)^{1-2\alpha} d\psi \leq  ct_{n}^{2\alpha-1}.
\end{split}
\end{equation}
From   \eqref{fractional resolvent estimate}, it yields
\begin{equation*}
\begin{split}
\|J_2 \|_{L^2(\Omega)}
&\leq c \|\partial^{k}_{t} f(0)\|_{L^2(\Omega)} \int^{\infty}_{\frac{\pi}{\tau\sin\theta}} e^{rt_{n}\cos\theta}r^{-k-1-\alpha}dr\\
&\leq c\tau^{k+1-\alpha} \|\partial^{k}_{t} f(0)\|_{L^2(\Omega)} \int^{\infty}_{\frac{\pi}{\tau\sin\theta}} e^{rt_{n}\cos\theta}r^{-2\alpha}dr
\leq  c\tau^{k+1-\alpha}t_{n}^{2\alpha-1}\|\partial^{k}_{t} f(0)\|_{L^2(\Omega)}.
\end{split}
\end{equation*}
Then  the desired result is obtained.
\end{proof}
\begin{lemma}\label{lemma3.10}
Let $V(t_{n})$ and $V^{n}$ be the solutions of \eqref{rfee} and \eqref{CLks}, respectively.
If $v=0$ and $f(t):=\frac{t^{k}}{k!}\ast\partial^{k+1}_{t} f(t)$, then
\begin{equation*}
\left\|V(t_{n})-V^{n}\right\|\leq c\tau^{k+1-\alpha}\int_{0}^{t_{n}}(t_{n}-s)^{2\alpha-1}
\left\|\partial^{k+1}_{s}f(s)\right\|ds.
\end{equation*}
\end{lemma}
\begin{proof}
By  \eqref{LT}, we obtain
\begin{equation}\label{nas3.6}
\begin{split}
V(t_{n})
&=\frac{1}{2\pi i}\int_{\Gamma_{\theta,\kappa}}{e^{zt_{n}}(z^{\alpha} - A)^{-1}\widehat{f}(z)}dz=(\mathscr{E}(t)\ast f(t))(t_{n})\\
&=\left(\mathscr{E}(t)\ast \left(\frac{t^{k}}{k!}\ast  \partial^{k+1}_{t}f(t)\right)\right)(t_{n})
=\left(\left(\mathscr{E}(t)\ast\frac{t^{k}}{k!}\right)\ast  \partial^{k+1}_{t}f(t)\right)(t_{n})
\end{split}
\end{equation}
with
\begin{equation}\label{nas3.007}
  \mathscr{E}(t)= \frac{1}{2\pi i}\int_{\Gamma_{\theta,\kappa}}{e^{zt}(z^{\alpha} - A)^{-1}}dz.
\end{equation}
From  \eqref{ads2.17}, it yields
\begin{equation*}
\begin{split}
\widetilde{V}(\xi)
&=\left(z^{\alpha}_{\tau} - A\right)^{-1}\widetilde{f}(\xi)
=\widetilde{\mathscr{E_{\tau}}}(\xi)\widetilde{f}(\xi)=\sum^{\infty}_{n=0}\mathscr{E}^{n}_{\tau}\xi^{n}\sum^{\infty}_{j=0}f^j\xi^{j}\\
&=\sum^{\infty}_{n=0}\sum^{\infty}_{j=0}\mathscr{E}^{n}_{\tau}f^j\xi^{n+j}=\sum^{\infty}_{j=0}\sum^{\infty}_{n=j}\mathscr{E}^{n-j}_{\tau}f^j\xi^{n}
=\sum^{\infty}_{n=0}\sum^{n}_{j=0}\mathscr{E}^{n-j}_{\tau}f^j\xi^{n}=\sum^{\infty}_{n=0}V^n\xi^{n}
\end{split}
\end{equation*}
with
\begin{equation*}
V^{n}=\sum^{n}_{j=0}\mathscr{E}^{n-j}_{\tau}f^j:=\sum^{n}_{j=0}\mathscr{E}^{n-j}_{\tau}f(t_{j}).
\end{equation*}
Here $\sum^{\infty}_{n=0}\mathscr{E}^{n}_{\tau}\xi^{n}=\widetilde{\mathscr{E_{\tau}}}(\xi)=\left( z^{\alpha}_{\tau} - A\right)^{-1}$
and by Cauchy's integral formula and the change of variables $\xi=e^{-z\tau}$ give the following representation for arbitrary $\rho \in (0,1)$
 \begin{equation*}
\mathscr{E}^{n}_{\tau}=\frac{1}{2\pi i}\int_{|\xi|=\rho}{\xi^{-n-1}\widetilde{\mathscr{E_{\tau}}}(\xi)}d\xi
=\frac{\tau}{2\pi i}\int_{\Gamma^{\tau}_{\theta,\kappa}}{e^{zn\tau}\left(z^{\alpha}_{\tau} - A\right)^{-1}}dz,
\end{equation*}
where  $\theta\in (\pi/2,\pi)$ is  sufficiently close to $\pi/2$ and $\kappa=t_{n}^{-1}$ in \eqref{Gamma}.

According to Lemma \ref{Lemma5} and \eqref{fractional resolvent estimate}, \eqref{ad3.3.09}, there exists
\begin{equation}\label{3.0002}
\|\mathscr{E}^{n}_{\tau}\| \leq c \tau \left( \int^{\frac{\pi}{\tau\sin\theta}}_{\kappa} e^{rt_{n}\cos\theta} r^{-\alpha}dr +\int^{\theta}_{-\theta}e^{\kappa t_{n}\cos\psi} \kappa^{1-\alpha}  d\psi\right)
\leq c\tau t_{n}^{\alpha-1}.
\end{equation}
Let $ \mathscr{E}_{\tau}(t)=\sum^{\infty}_{n=0}\mathscr{E}^{n}_{\tau}\delta_{t_{n}}(t)$, with $\delta_{t_{n}}$ being the Dirac delta function at $t_{n}$.
Then
\begin{equation}\label{nas3.8}
\begin{split}
  (\mathscr{E}_{\tau}(t)\ast f(t))(t_{n})
  &=\left(\sum^{\infty}_{j=0}\mathscr{E}^{j}_{\tau}\delta_{t_{j}}(t) \ast f(t) \right)(t_{n})\\
  &=\sum^{n}_{j=0}\mathscr{E}^{j}_{\tau}f(t_{n}-t_{j})
  =\sum^{n}_{j=0}\mathscr{E}^{n-j}_{\tau}f(t_{j})=V^{n}.
\end{split}
\end{equation}
Moreover, using the above equation and  \eqref{DLTcoeff}, there exist
\begin{equation*}
\begin{split}
  \widetilde{(\mathscr{E}_{\tau}\ast t^{k})}(\xi)
& = \sum^{\infty}_{n=0} \sum^{n}_{j=0}\mathscr{E}^{n-j}_{\tau}t_{j}^{k}\xi^{n}
  =\sum^{\infty}_{j=0} \sum^{\infty}_{n=j}\mathscr{E}^{n-j}_{\tau}t_{j}^{k}\xi^{n}
  =\sum^{\infty}_{j=0} \sum^{\infty}_{n=0}\mathscr{E}^{n}_{\tau}t_{j}^{k}\xi^{n+j}\\
& =\sum^{\infty}_{n=0}\mathscr{E}^{n}_{\tau}\xi^{n}\sum^{\infty}_{j=0}t_{j}^{k}\xi^{j}
  =\widetilde{\mathscr{E_{\tau}}}(\xi) \tau^{k}\sum^{\infty}_{j=0}j^{k}\xi^{j}
  =\widetilde{\mathscr{E}_{\tau}}(\xi) \tau^{k}\gamma_{k}(\xi).
\end{split}
\end{equation*}
Using  \eqref{nas3.6}, \eqref{nas3.8} and  Lemma \ref{lemma3.9}, we have the following estimate
\begin{equation}\label{nad3.10}
\left\|\left((\mathscr{E}_{\tau}-\mathscr{E}) \ast \frac{t^{k}}{k!}\right)(t_{n})\right\| \leq c\tau^{k+1}t_{n}^{\alpha-1}.
\end{equation}

Next, we prove the following inequality \eqref{3.0003}  for $t>0$
\begin{equation}\label{3.0003}
\left\|\left((\mathscr{E}_{\tau}-\mathscr{E}) \ast \frac{t^{k}}{k!}\right)(t)\right\| \leq c\tau^{k+1}t^{\alpha-1},\quad \forall t\in (t_{n-1},t_{n}).
\end{equation}
By Taylor series expansion of $\mathscr{E}(t)$ at $t=t_{n}$, we get
\begin{equation*}
\begin{split}
 \left( \mathscr{E} \ast \frac{t^{k}}{k!}\right)(t)
=&\left( \mathscr{E} \ast \frac{t^{k}}{k!}\right)(t_{n})+(t-t_{n})\left( \mathscr{E} \ast \frac{t^{k-1}}{(k-1)!}\right)(t_{n})+\cdots + \frac{(t-t_{n})^{k-1}}{(k-1)!}\left( \mathscr{E} \ast t \right)(t_{n}) \\
 & + \frac{(t-t_{n})^{k}}{k!}\left( \mathscr{E} \ast 1 \right)(t_{n})
 +\frac{1}{k!}\int^{t}_{t_{n}}(t-s)^{k}\mathscr{E}(s)ds,
\end{split}
\end{equation*}
which  also holds  for $  \left( \mathscr{E}_{\tau} \ast \frac{t^{k}}{k!}\right)(t) $.
Therefore, using  \eqref{nad3.10},  it yields
\begin{equation*}
\left\|\left((\mathscr{E}_{\tau}-\mathscr{E}) \ast \frac{t^{l}}{l!}\right)(t_{n})\right\| \leq c\tau^{l+1}t_{n}^{\alpha-1}.
\end{equation*}
According to  \eqref{nas3.007}, \eqref{fractional resolvent estimate} and \eqref{ad3.3.09}, one has
\begin{equation*}
\begin{split}
\| \mathscr{E}(t) \|
\leq c \left( \int^{\infty}_{\kappa}e^{rt\cos\theta}r^{-\alpha}dr + \int^{\theta}_{-\theta}e^{\kappa t \cos\psi}\kappa^{1-\alpha}d\psi \right)
\leq c t^{\alpha-1}.
\end{split}
\end{equation*}
Moreover, we get
\begin{equation*}
\left\| \int^{t}_{t_{n}}(t-s)^{k}\mathscr{E}(s)ds \right\| \leq c \int^{t_{n}}_{t}(s-t)^{k}s^{\alpha-1}ds \leq c \tau^{k+1} t^{\alpha-1}.
\end{equation*}
According to the definition of $ \mathscr{E}_{\tau}(t)=\sum^{\infty}_{n=0}\mathscr{E}^{n}_{\tau}\delta_{t_{n}}(t)$ in \eqref{nas3.8} and \eqref{3.0002}, we deduce
\begin{equation*}
\left\| \int^{t}_{t_{n}}(t-s)^{k}\mathscr{E}_{\tau}(s)ds \right\|\leq  (t_n-t)^{k} \| \mathscr{E}^{n}_{\tau} \|  \leq c \tau^{k+1} t_{n}^{\alpha-1} \leq c \tau^{k+1} t^{\alpha-1}~~~ \forall t\in (t_{n-1},t_{n}).
\end{equation*}
By \eqref{nad3.10} and the above inequalities, it yields the inequality \eqref{3.0003}.
The proof is completed.
\end{proof}

\begin{theorem}
Let $V(t_{n})$ and $V^{n}$ be the solutions of \eqref{rfee} and \eqref{CLks}, respectively. Let $v\in L^{2}(\Omega)$, $f\in C^{k}([0,T]; L^{2}(\Omega))$ and $\int_{0}^{t} (t-s)^{2\alpha-1} \|\partial^{k+1}_{s}f(s) \| ds <\infty$ with $1\leq k \leq6$.  Then
\begin{equation*}
\begin{split}
   &\left\|V^{n}-V(t_{n})\right\|\\
   &\quad \leq C\tau^{k+1-\alpha}
     \left( t^{\alpha-(k+1)}_{n}\|v\| + \sum_{l=0}^{k} t^{2\alpha+l-(k+1)}_{n}\left\| \partial^{l}_{t} f(0) \right\|
 + \int_{0}^{t_{n}}(t_{n}-s)^{2\alpha-1} \|\partial^{k+1}_{s}f(s) \| ds \right).
\end{split}
\end{equation*}
\end{theorem}
\begin{proof}
Subtracting \eqref{LT} from \eqref{DLT}, we obtain
\begin{equation*}
V^{n}-V(t_{n})=I_{1}+\sum_{l=1}^{k-1}I_{2,l}+I_{3}-I_{4},
\end{equation*}
where
\begin{equation*}
\begin{split}
I_{1} =& \frac{1}{2\pi i} \int_{\Gamma^{\tau}} e^{zt_{n}} \left[ \bar{\mu}(e^{-z\tau})K(z_{\tau}) - K(z) \right] (Av+f(0)) dz, \\
I_{2,l} =& \frac{1}{2\pi i} \int_{\Gamma^{\tau}}e^{zt_{n}}\left[z_{\tau}\left(\frac{\gamma_{l}}{l!}+\sum_{j=1}^{k}d^{(k)}_{l,j}\xi^{j}\right)\tau^{l+1} K(z_{\tau})-z^{-l}K(z)\right]\partial^{l}_{t}f(0)dz, \\
\end{split}
\end{equation*}
and
\begin{equation*}
\begin{split}
I_{3} =& \frac{1}{2\pi i} \int_{\Gamma^{\tau}} e^{zt_{n}} z_{\tau}K(z_{\tau}) \tau \widetilde{R_{k}}(e^{-z\tau}) dz - \frac{1}{2\pi i} \int_{\Gamma} e^{zt_{n}} zK(z)\widehat{R_{k}}(z)  dz, \\
I_{4} =& \frac{1}{2\pi i} \int_{\Gamma\backslash\Gamma^{\tau}} e^{zt_{n}}K(z) \left[ Av+f(0) + \sum_{l=1}^{k-1}z^{-l}\partial^{l}_{t}f(0)  \right]  dz.
\end{split}
\end{equation*}
Using the triangle inequality, Lemma \ref{Lemma2}, \ref{Lemma3} and $\kappa=t_{n}^{-1}$, we estimate the first term
\begin{equation*}
\begin{split}
\left\|I_{1}\right\|
\leq & C \int_{\Gamma^{\tau}}\left|e^{zt_{n}}\right| \left( \left\|\bar{\mu}(e^{-z\tau})K(z_{\tau})A-K(z)A\right\| \|v\| + \left\|\bar{\mu}(e^{-z\tau})K(z_{\tau})-K(z)\right\| \|f(0)\| \right) |dz|\\
\leq &  C\tau^{k+1-\alpha} \int_{\Gamma^{\tau}}\left|e^{zt_{n}}\right| \left( |z|^{k-\alpha} \|v\| +  |z|^{k-2\alpha} \|f(0)\| \right)|dz|\\
 =& C\tau^{k+_1-\alpha}\left( t_{n}^{\alpha-(k+1)} \int^{\frac{t_{n}\pi}{\tau\sin\theta}}_{1} e^{rt_{n}\cos\theta}  (rt_{n})^{k-\alpha} d(rt_{n}) \|v\|
    +t_{n}^{\alpha-(k+1)}\int^{\theta}_{-\theta}e^{\cos\varphi} d\varphi\|v\| \right. \\
  & \left.+t_{n}^{2\alpha-(k+1)}\int^{\frac{t_{n}\pi}{\tau\sin\theta}}_{1} e^{rt_{n}\cos\theta} (rt_{n})^{k-2\alpha} d(rt_{n})\|f(0)\| +t_{n}^{2\alpha-(k+1)}\int^{\theta}_{-\theta}e^{\cos\varphi} d\varphi\|f(0)\|\right) \\
\leq & C\tau^{k+1-\alpha}\left( t_{n}^{\alpha-(k+1)} \|v\| +t_{n}^{2\alpha-(k+1)} \|f(0)\| \right)
\end{split}
\end{equation*}

Noting $(z^{\alpha}_{\tau}+A)^{-1} -(z^{\alpha}+A)^{-1}=(z^{\alpha}_{\tau}+A)^{-1}(z^{\alpha}+A)^{-1} ( z^{\alpha}-z^{\alpha}_{\tau}),$
 and according to   the triangle inequality, \eqref{fractional resolvent estimate} and  Lemmas  \ref{Lemma4}-\ref{Lemma5},
we estimate the second term
\begin{equation*}
\begin{split}
\|I_{2,l}\|
\leq & \frac{1}{2\pi} \int_{\Gamma^{\tau}} \left|e^{zt_{n}}\right| \left|\left(\frac{\gamma_{l}}{l!}+\sum_{j=1}^{k}d^{(k)}_{l,j}\xi^{j}\right)\tau^{l+1} - \frac{1}{z^{l+1}} \right| \left\|(z^{\alpha}_{\tau}+A)^{-1} \right\|
 \left\|\partial^{l}_{t}f(0)\right\|  |dz| \\
 & +\frac{1}{2\pi} \int_{\Gamma^{\tau}} \left|e^{zt_{n}}\right| |z|^{-l-1} \left\|(z^{\alpha}_{\tau}+A)^{-1} -(z^{\alpha}+A)^{-1}\right\| \left\|\partial^{l}_{t}f(0)\right\|  |dz| \\
\leq & C\int_{\Gamma^{\tau}} \left|e^{zt_{n}}\right| \tau^{k+1}|z|^{k-l-\alpha}\left\|\partial^{l}_{t}f(0)\right\|  |dz|
  + C\int_{\Gamma^{\tau}} \left|e^{zt_{n}}\right| \tau^{k+1-\alpha}|z|^{k-l-2\alpha} \left\|\partial^{l}_{t}f(0)\right\|  |dz| \\
\leq & C \tau^{k+1-\alpha} \int_{\Gamma^{\tau}} \left|e^{zt_{n}}\right| |z|^{k-l-2\alpha}  |dz| \left\| \partial^{l}_{t}f(0)\right\|
\leq  C \tau^{k+1-\alpha}  t_{n}^{2\alpha+l-(k+1)} \left\| \partial^{l}_{t}f(0)\right\|.
\end{split}
\end{equation*}

We next estimate $I_{4}$ as following. From the triangle inequality and the resolvent estimate \eqref{fractional resolvent estimate}, it yields
\begin{equation*}
\begin{split}
\left\|I_{4}\right\|
 \leq& C \int_{\Gamma\backslash\Gamma^{\tau}}\left|e^{zt_{n}}\right| \left\| K(z)A \right\|\left\|v\right\|dz+C\int_{\Gamma\backslash\Gamma^{\tau}} \left|e^{zt_{n}}\right|\left\|K(z)\right\|\left\|f(0)\right\| |dz| \\
 & + C \int_{\Gamma\backslash\Gamma^{\tau}} \left|e^{zt_{n}}\right| \left\| K(z) \right\|  \sum_{l=1}^{k-1}|z|^{-l} \left\|\partial^{l}_{t}f(0) \right\|  |dz|  \\
 \leq& C \int_{\Gamma\backslash\Gamma^{\tau}}\left|e^{zt_{n}}\right| |z|^{-1} \left\|v\right\|dz+C\int_{\Gamma\backslash\Gamma^{\tau}} \left|e^{zt_{n}}\right| |z|^{-1-\alpha} \left\|f(0)\right\| |dz| \\
 & + C \int_{\Gamma\backslash\Gamma^{\tau}} \left|e^{zt_{n}}\right|  \sum_{l=1}^{k-1}|z|^{-l-1-\alpha} \left\|\partial^{l}_{t}f(0) \right\|   |dz| \\
  \leq& C \tau^{k+1-\alpha}\left( t_{n}^{\alpha-(k+1)} \left\|v\right\| + t_{n}^{2\alpha-(k+1)}  \left\|f(0)\right\| |dz|
  +  \sum_{l=1}^{k-1} t_{n}^{2\alpha+l-(k+1)}  \left\|\partial^{l}_{t}f(0) \right\| \right).
\end{split}
\end{equation*}
Since
\begin{equation*}
\begin{split}
\int_{\Gamma\backslash\Gamma^{\tau}}\left|e^{zt_{n}}\right| |z|^{-1}  |dz|  \left\|v\right\|
= &\int^{\infty}_{\frac{\pi}{\tau\sin\theta}}\left|e^{zt_{n}}\right||z|^{-1}|dz|\left\|v\right\| \leq  \int^{\infty}_{\frac{\pi}{\tau\sin\theta}} e^{rt_{n}\cos\theta}  r^{-1}  dr \left\|v\right\| \\
\leq  &  \tau^{k+1-\alpha} \int^{\infty}_{\frac{\pi}{\tau\sin\theta}} e^{rt_{n}\cos\theta} r^{k-\alpha} dr \left\|v\right\| \leq  \tau^{k+1-\alpha} t_{n}^{\alpha-(k+1)} \left\|v\right\|
\end{split}
\end{equation*}
with $ 1\leq \left( \frac{\sin\theta}{\pi} \right)^{k+1-\alpha} \tau^{k+1-\alpha} r^{k+1-\alpha} \leq  \tau^{k+1-\alpha} r^{k+1-\alpha}$, $r \geq \frac{\pi}{\tau\sin\theta}$.

According to Lemmas \ref{lemma3.9} and \ref{lemma3.10} with $R_{k} = \frac{t^{k}}{k!}\partial^{k}_{t} f(0)+\frac{t^{k}}{k!} \ast \partial^{k+1}_{t} f(t)$, there exist
\begin{equation*}
\left\|I_{3}\right\|\leq c\tau^{k+1}t_{n}^{\alpha-1} \left\| \partial^{k}_{t} f(0) \right\|+ c\tau^{k+1} \int_{0}^{t_{n}}(t_{n}-s)^{2\alpha-1} \left\|\partial^{k+1}_{s}f(s) \right\| ds.
\end{equation*}
The proof is completed.
\end{proof}

\section{Numerical results}\label{Se:numer}
We  numerically verify the above theoretical results  and the discrete
$L^2$-norm is used to measure the numerical errors.
In the space direction, it is discretized with the   spectral collocation method  with the Chebyshev-Gauss-Lobatto points \cite{ACYZ:21,STW:2011}.
Here we main focus on the time direction convergence order, since the convergence rate of the spatial discretization is well understood.

\begin{example}
Let us consider the following  subdiffusion \eqref{fee}
\begin{equation*}
\begin{split}
& ^{C}_{0}D^{\alpha}_{t}u(x,t) - \frac{\partial^{2}u(x,t)}{\partial x^{2}} = f(x,t),  \\
& u(-1,t)=u(1,t)=0, \\
& u(x,0)=v(x)
\end{split}
\end{equation*}
with the nonsmooth data $v(x)=\sqrt{1-x^2}$ and $f(x,t)=(t+1)^8\left(1+\chi_{(0,1)}(x)\right)$. Here
\begin{equation*}
\chi_{(0,1)}(x)=\left\{ \begin{array}
{l@{\quad} l}1,~~~~0<x<1,\\
0,~~~~{\rm elsewhere}.
\end{array}
\right.
\end{equation*}
\end{example}
Since the analytic solutions is unknown, the order of the convergence of the numerical results are computed by the following formula
\begin{equation*}
  {\rm Convergence ~Rate}=\frac{\ln \left(||u^{N/2}-u^{N}||/||u^{N}-u^{2N}||\right)}{\ln 2}
\end{equation*}
with $u^N=V^N+v$ in  \eqref{CLks}.

\begin{table}[!ht]
\begin{center}
\caption{The  errors and convergent order of uncorrection $L_k$ approximation \eqref{SLks2}.}%%%, k=4,5,6
\begin{tabular}{c c c c c c c c}
\hline
         $k$                 &  $\alpha$ &  $N=80$      & $N=160$    &  $N=320$     &  $N=640$   &  $N=1280$   &  Rate              \\ \hline
 \multirow{3}{*}{4}          &  0.2      &  4.8406e-04  & 2.3850e-04 &  1.1712e-04  & 5.7975e-05 &  2.8839e-05 &  $\approx$ 1.0173  \\
                             &  0.5      &  1.0889e-03  & 5.6592e-04 &  2.7928e-04  & 1.3829e-04 &  6.8792e-05 &  $\approx$ 0.9961  \\
                             &  0.8      &  1.4043e-03  & 8.8681e-04 &  4.4652e-04  & 2.2157e-04 &  1.1023e-04 &  $\approx$ 0.9178  \\ \hline
\multirow{3}{*}{5}           &  0.2      &  4.9653e-04  & 2.3900e-04 &  1.1713e-04  & 5.7975e-05 &  2.8840e-05 &  $\approx$ 1.0265  \\
                             &  0.5      &  1.1835e-03  & 5.7040e-04 &  2.7948e-04  & 1.3830e-04 &  6.8792e-05 &  $\approx$ 1.0261  \\
                             &  0.8      &  1.8875e-03  & 9.1481e-04 &  4.4809e-04  & 2.2166e-04 &  1.1023e-04 &  $\approx$ 1.0245  \\ \hline
\multirow{3}{*}{6}           &  0.2      &  4.9694e-04  & 2.3901e-04 &  1.1714e-04  & 5.7975e-05 &  2.8840e-05 &  $\approx$ 1.0268  \\
                             &  0.5      &  1.2767e-03  & 5.7052e-04 &  2.7949e-04  & 1.3830e-04 &  6.8792e-05 &  $\approx$ 1.0535  \\
                             &  0.8      &  1.6587e-03  & 9.1580e-04 &  4.4811e-04  & 2.2166e-04 &  1.1023e-04 &  $\approx$ 0.9778  \\ \hline
\end{tabular}
\label{Ta:tableLk1}
\end{center}
\end{table}
\begin{table}[!ht]
\begin{center}
\caption{The  errors and convergent order of correction $L_k$ approximation  \eqref{CLks}.}
\begin{tabular}{c c c c c c c c}
\hline
        $k$                  &  $\alpha$ &  $N=20$    &  $N=40$     &  $N=80$    &  $N=160$     &  $N=320$   &  Rate                    \\ \hline
 \multirow{3}{*}{4}          &  0.2      & 7.9510e-03 & 3.1598e-04   & 1.2879e-05  & 5.0269e-07   & 1.9150e-08  &  $\approx$ 4.6658(4.8)    \\
                             &  0.5      & 3.4422e-02 & 2.0224e-03   & 9.7963e-05  & 4.5534e-06   & 2.0704e-07  &  $\approx$ 4.3358(4.5)     \\
                             &  0.8      & 1.3454e-01 & 8.5003e-03   & 5.0322e-04  & 2.8569e-05   & 1.5882e-06  &  $\approx$ 4.0926(4.2)      \\ \hline
 \multirow{3}{*}{5}          &  0.2      & 6.0454e-03 & 9.2733e-06   & 4.1798e-07  & 8.0735e-09   & 1.5835e-10  &  $\approx$ 6.2965(5.8)       \\
                             &  0.5      & 1.2783e-02 & 1.8653e-04   & 3.2579e-06  & 7.9314e-08   & 1.8079e-09  &  $\approx$ 5.6884(5.5)        \\
                             &  0.8      & 1.9116e-02 & 2.8855e-04   & 2.0050e-05  & 5.5691e-07   & 1.5410e-08  &  $\approx$ 5.0606(5.2)         \\ \hline
 \multirow{3}{*}{6}          &  0.2      & 1.3275e-01 & 1.1518e-03   & 2.4273e-08  & 7.6887e-11   & 1.4921e-11  &  $\approx$ 8.2626(6.8)          \\
                             &  0.5      & 3.4652e-01 & 2.7804e-02   & 4.2374e-04  & 2.9762e-07   & 1.2242e-10  &  $\approx$ 7.8494(6.5)           \\
                             &  0.8      & 2.5905e-01 & 2.8118e-02   & 1.4048e-03  & 2.8657e-06   & 2.6575e-10  &  $\approx$ 7.4650(6.2)            \\ \hline
\end{tabular}
\label{Ta:tableL4113}
\end{center}
\end{table}

Table \ref{Ta:tableLk1} shows that the stand   $L_k$ approximation  in \eqref{SLks2} just  achieves the  first-order convergence. However,
the correction $L_k$  in \eqref{CLks} preserves the high-order convergence rate with  nonsmooth data in Table \ref{Ta:tableL4113}.

\section*{Appendix}\label{App:coeff}

The coefficients $\omega^{(k)}_{j}$ of $L_k$ approximation  in \eqref{add2.5} are given explicitly by the following
{\tiny{
\begin{itemize}
\item $L_1$ approximation
 \begin{equation*}\label{rcoeff1}
\begin{split}
\omega^{(1)}_{0}=\frac{1}{\Gamma(2-\alpha)}, \quad
\omega^{(1)}_{j}=\frac{(j+1)^{1-\alpha}-2j^{1-\alpha}+(j-1)^{1-\alpha}}{\Gamma(2-\alpha)},~~j\geq 1.
\end{split}
\end{equation*}
\end{itemize}

\begin{itemize}
\item $L_2$ approximation
\begin{equation*}
\begin{split}
\omega^{(2)}_{0}=\frac{1}{\Gamma(3-\alpha)} +\frac{1}{2} \frac{1}{\Gamma(2-\alpha)}, \quad
\omega^{(2)}_{1}= \frac{\left(2^{2-\alpha}-3\right)}{\Gamma(3-\alpha)}  +\frac{1}{2}\frac{\left(2^{1-\alpha}-3\right)}{\Gamma(2-\alpha)},
\end{split}
\end{equation*}
\begin{equation*}
\begin{split}
\omega^{(2)}_{j}
= & \frac{\left((j+1)^{2-\alpha}-3j^{2-\alpha}+3(j-1)^{2-\alpha}-(j-2)^{2-\alpha}\right)}{\Gamma(3-\alpha)}  +\frac{1}{2}\frac{\left((j+1)^{1-\alpha}-3j^{1-\alpha}+3(j-1)^{1-\alpha}-(j-2)^{1-\alpha}\right)}{\Gamma(2-\alpha)},~~j\geq 2.
\end{split}
\end{equation*}
\end{itemize}

\begin{itemize}
\item $L_3$ approximation
\begin{equation*}
\begin{split}
\omega^{(3)}_{0}=\frac{1}{\Gamma(4-\alpha)} + \frac{1}{\Gamma(3-\alpha)}+\frac{1}{3}\frac{1}{\Gamma(2-\alpha)}, \quad
\omega^{(3)}_{1}= \frac{ 2^{3-\alpha}-4 }{\Gamma(4-\alpha)}  + \frac{ 2^{2-\alpha}-4 }{\Gamma(3-\alpha)}+\frac{1}{3}\frac{ 2^{1-\alpha}-4 }{\Gamma(2-\alpha)},
\end{split}
\end{equation*}
\begin{equation*}
\begin{split}
\omega^{(3)}_{2}
= & \frac{ 3^{3-\alpha}-4\times 2^{3-\alpha}+6 }{\Gamma(4-\alpha)}  + \frac{ 3^{2-\alpha}-4\times 2^{2-\alpha}+6 }{\Gamma(3-\alpha)}
   +\frac{1}{3}\frac{ 3^{1-\alpha}-4\times 2^{1-\alpha}+6 }{\Gamma(2-\alpha)},
\end{split}
\end{equation*}
\begin{equation*}
\begin{split}
\omega^{(3)}_{j}=
& \frac{ (j+1)^{3-\alpha}-4j^{4-\alpha}+6(j-1)^{3-\alpha}-4(j-2)^{3-\alpha}
+(j-3)^{3-\alpha} }{\Gamma(4-\alpha)}
  + \frac{ (j+1)^{2-\alpha}-4j^{2-\alpha}+6(j-1)^{2-\alpha}-4(j-2)^{2-\alpha}
+(j-3)^{2-\alpha} }{\Gamma(3-\alpha)}       \\
 & +\frac{1}{3}\frac{ (j+1)^{1-\alpha}-4j^{1-\alpha}+6(j-1)^{1-\alpha}-4(j-2)^{1-\alpha}
+(j-3)^{1-\alpha} }{\Gamma(2-\alpha)},~j\geq 3.
\end{split}
\end{equation*}
\end{itemize}

\begin{itemize}
\item $L_4$ approximation
\begin{equation*}
\begin{split}
\omega^{(4)}_{0}=\frac{1}{\Gamma(5-\alpha)} +\frac{3}{2} \frac{1}{\Gamma(4-\alpha)}+\frac{11}{12}\frac{1}{\Gamma(3-\alpha)} +\frac{1}{4}\frac{1}{\Gamma(2-\alpha)}, \quad
\omega^{(4)}_{1}=
  \frac{ 2^{4-\alpha}-5 }{\Gamma(5-\alpha)}  +\frac{3}{2}\frac{ 2^{3-\alpha}-5 }{\Gamma(4-\alpha)}
  +\frac{11}{12}\frac{ 2^{2-\alpha}-5 }{\Gamma(3-\alpha)}+\frac{1}{4}\frac{ 2^{1-\alpha}-5 }{\Gamma(2-\alpha)},
  \end{split}
\end{equation*}
\begin{equation*}
\begin{split}
\omega^{(4)}_{2}
=  \frac{ 3^{4-\alpha}-5\times 2^{4-\alpha}+10 }{\Gamma(5-\alpha)}  +\frac{3}{2}\frac{ 3^{3-\alpha}-5\times 2^{3-\alpha}+10 }{\Gamma(4-\alpha)}
   +\frac{11}{12}\frac{ 3^{2-\alpha}-5\times 2^{2-\alpha}+10 }{\Gamma(3-\alpha)} +\frac{1}{4}\frac{ 3^{1-\alpha}-5\times 2^{1-\alpha}+10 }{\Gamma(2-\alpha)},
\end{split}
\end{equation*}
\begin{equation*}
\begin{split}
\omega^{(4)}_{3}
= & \frac{ 4^{4-\alpha}-5\times 3^{4-\alpha}+10\times2^{4-\alpha} -10  }{\Gamma(5-\alpha)}   + \frac{3}{2}\frac{ 4^{3-\alpha}-5\times 3^{3-\alpha}+10\times2^{3-\alpha} -10 }{\Gamma(4-\alpha)}   \\
  &+\frac{11}{12}\frac{ 4^{2-\alpha}-5\times 3^{2-\alpha}+10\times2^{2-\alpha}-10 }{\Gamma(3-\alpha)} +\frac{1}{4}\frac{ 4^{1-\alpha}-5\times 3^{1-\alpha}+10\times2^{1-\alpha}-10  }{\Gamma(2-\alpha)},
\end{split}
\end{equation*}
\begin{equation*}
\begin{split}
\omega^{(4)}_{j}=& \frac{  (j+1)^{4-\alpha}-5j^{4-\alpha}+10(j-1)^{4-\alpha} -10(j-2)^{4-\alpha}  +5(j-3)^{4-\alpha}-(j-4)^{4-\alpha}    }{\Gamma(5-\alpha)}              \\
& +\frac{3}{2}\frac{ (j+1)^{3-\alpha}-5j^{3-\alpha}+10(j-1)^{3-\alpha}-10(j-2)^{3-\alpha}  +5(j-3)^{3-\alpha}-(j-4)^{3-\alpha} }{\Gamma(4-\alpha)}       \\
&+\frac{11}{12}\frac{ (j+1)^{2-\alpha}-5j^{2-\alpha}+10(j-1)^{2-\alpha}  -10(j-2)^{2-\alpha}  +5(j-3)^{2-\alpha}-(j-4)^{2-\alpha} }{\Gamma(3-\alpha)}   \\
& +\frac{1}{4}\frac{ (j+1)^{1-\alpha}-5j^{1-\alpha}+10(j-1)^{1-\alpha} -10(j-2)^{1-\alpha}  +5(j-3)^{1-\alpha}-(j-4)^{1-\alpha} }{\Gamma(2-\alpha)},~j\geq 4.
\end{split}
\end{equation*}
\end{itemize}

\begin{itemize}
\item $L_5$ approximation
\begin{equation*}
\begin{split}
\omega^{(5)}_{0}=
&\frac{1}{\Gamma(6-\alpha)} + 2 \frac{1}{\Gamma(5-\alpha)} + \frac{7}{4}\frac{1}{\Gamma(4-\alpha)} + \frac{5}{6}\frac{1}{\Gamma(3-\alpha)} + \frac{1}{5}\frac{1}{\Gamma(2-\alpha)},
\end{split}
\end{equation*}
\begin{equation*}
\begin{split}
\omega^{(5)}_{1}=
& \frac{ 2^{5-\alpha}-6  }{\Gamma(6-\alpha)} + 2\frac{ 2^{4-\alpha}-6 }{\Gamma(5-\alpha)} + \frac{7}{4}\frac{ 2^{3-\alpha}-6 }{\Gamma(4-\alpha)}
  +\frac{5}{6}\frac{ 2^{2-\alpha}-6 }{\Gamma(3-\alpha)} + \frac{1}{5}\frac{ 2^{1-\alpha}-6 }{\Gamma(2-\alpha)},
\end{split}
\end{equation*}
\begin{equation*}
\begin{split}
\omega^{(5)}_{2}=
& \frac{ 3^{5-\alpha}-6\times2^{5-\alpha}+15 }{\Gamma(6-\alpha)} + 2\frac{ 3^{4-\alpha}-6\times2^{4-\alpha}+15 }{\Gamma(5-\alpha)}
  +\frac{7}{4}\frac{ 3^{3-\alpha}-6\times2^{3-\alpha}+15 }{\Gamma(4-\alpha)}
 +\frac{5}{6}\frac{ 3^{2-\alpha}-6\times 2^{2-\alpha}+15 }{\Gamma(3-\alpha)} + \frac{1}{5}\frac{ 3^{1-\alpha}-6\times 2^{1-\alpha}+15 }{\Gamma(2-\alpha)},
\end{split}
\end{equation*}
\begin{equation*}
\begin{split}
\omega^{(5)}_{3}=
& \frac{ 4^{5-\alpha}-6\times 3^{5-\alpha}+15\times2^{5-\alpha} -20  }{\Gamma(6-\alpha)} +2\frac{ 4^{4-\alpha}-6\times 3^{4-\alpha}+15\times2^{4-\alpha} -20 }{\Gamma(5-\alpha)} +\frac{7}{4}\frac{ 4^{3-\alpha}-6\times 3^{3-\alpha}+15\times2^{3-\alpha}-20 }{\Gamma(4-\alpha)}  \\
& +\frac{5}{6}\frac{ 4^{2-\alpha}-6\times 3^{2-\alpha}+15\times2^{2-\alpha}-20 }{\Gamma(3-\alpha)}
 +\frac{1}{5}\frac{ 4^{1-\alpha}-6\times 3^{1-\alpha}+15\times2^{1-\alpha} -20 }{\Gamma(2-\alpha)},
\end{split}
\end{equation*}
\begin{equation*}
\begin{split}
\omega^{(5)}_{4}=
& \frac{5^{5-\alpha}-6\times 4^{5-\alpha}+15\times3^{5-\alpha} -20 \times2^{5-\alpha} +15  }{\Gamma(6-\alpha)}
 +2\frac{5^{4-\alpha}-6\times 4^{4-\alpha}+15\times3^{4-\alpha} -20 \times2^{4-\alpha} +15}{\Gamma(5-\alpha)}                \\
& +\frac{7}{4}\frac{5^{3-\alpha}-6\times 4^{3-\alpha}+15\times3^{3-\alpha} -20 \times2^{3-\alpha} +15 }{\Gamma(4-\alpha)}
 +\frac{5}{6}\frac{5^{2-\alpha}-6\times 4^{2-\alpha}+15\times3^{2-\alpha} -20 \times2^{2-\alpha} +15}{\Gamma(3-\alpha)}      \\
& +\frac{1}{5}\frac{(5^{1-\alpha}-6\times 4^{1-\alpha} +15\times3^{1-\alpha} -20 \times2^{1-\alpha} +15 }{\Gamma(2-\alpha)},
\end{split}
\end{equation*}
\begin{equation*}
\begin{split}
\omega^{(5)}_{j}
=& \frac{(j+1)^{5-\alpha}-6j^{5-\alpha}+15(j-1)^{5-\alpha} -20(j-2)^{5-\alpha}  +15(j-3)^{5-\alpha}-6(j-4)^{5-\alpha} +(j-5)^{5-\alpha}}{\Gamma(6-\alpha)} \\
&+2\frac{(j+1)^{4-\alpha}-6j^{4-\alpha}+15(j-1)^{4-\alpha} -20(j-2)^{4-\alpha} +15(j-3)^{4-\alpha}-6(j-4)^{4-\alpha} +(j-5)^{4-\alpha}}{\Gamma(5-\alpha)}   \\
&+\frac{7}{4}\frac{(j+1)^{3-\alpha}-6j^{3-\alpha}+15(j-1)^{3-\alpha}-20(j-2)^{3-\alpha}+15(j-3)^{3-\alpha}-6(j-4)^{3-\alpha}+(j-5)^{3-\alpha}}{\Gamma(4-\alpha)}\\
&+\frac{5}{6}\frac{(j+1)^{2-\alpha}-6j^{2-\alpha}+15(j-1)^{2-\alpha}-20(j-2)^{2-\alpha}+15(j-3)^{2-\alpha}-6(j-4)^{2-\alpha} +(j-5)^{2-\alpha}  }{\Gamma(3-\alpha)}   \\
&+\frac{1}{5}\frac{(j+1)^{1-\alpha}-6j^{1-\alpha}+15(j-1)^{1-\alpha} -20(j-2)^{1-\alpha}+15(j-3)^{1-\alpha}-6(j-4)^{1-\alpha} +(j-5)^{1-\alpha}  }{\Gamma(2-\alpha)},~j\geq 5.
\end{split}
\end{equation*}
\end{itemize}
\begin{itemize}
\item $L_6$ approximation
\begin{equation*}
\begin{split}
\omega^{(6)}_{0} = & \frac{1}{\Gamma(7-\! \alpha)} \!+\! \frac{5}{2}\frac{1}{\Gamma(6-\!\alpha)} \!+\! \frac{17}{6}\frac{1}{\Gamma(5-\!\alpha)} \!+\! \frac{15}{8}\frac{1}{\Gamma(4-\!\alpha)} \!+\!\frac{137}{180}\frac{1}{\Gamma(3-\!\alpha)} \!+\! \frac{1}{6}\frac{1}{\Gamma(2-\!\alpha)},
\end{split}
\end{equation*}
\begin{equation*}
\begin{split}
\omega^{(6)}_{1}\!\!=\!
& \frac{\ 2^{6-\alpha}-7 }{\Gamma(7-\!\alpha)}  \!+\! \frac{5}{2}\frac{2^{5-\alpha}-7}{\Gamma(6-\!\alpha)}  \!+\! \frac{17}{6}\frac{ 2^{4-\alpha}-7}{\Gamma(5-\!\alpha)} \!+\! \frac{15}{8}\frac{2^{3-\alpha}-7}{\Gamma(4-\!\alpha)}  \!+\! \frac{137}{180}\frac{2^{2-\alpha}-7}{\Gamma(3-\!\alpha)} \!+\! \frac{1}{6}\frac{ 2^{1-\alpha}-7}{\Gamma(2-\!\alpha)},
\end{split}
\end{equation*}
\begin{equation*}
\begin{split}
\omega^{(6)}_{2}
=& \frac{ 3^{6-\alpha}-7\times 2^{6-\alpha}+21  }{\Gamma(7-\alpha)}  +\frac{5}{2}\frac{ 3^{5-\alpha}-7\times 2^{5-\alpha}+21 }{\Gamma(6-\alpha)}
  +\frac{17}{6}\frac{ 3^{4-\alpha}-7\times 2^{4-\alpha}+21  }{\Gamma(5-\alpha)}  \\
&+\frac{15}{8}\frac{  3^{3-\alpha}-7\times 2^{3-\alpha}+21  }{\Gamma(4-\alpha)}  +\frac{137}{180}\frac{ 3^{2-\alpha}-7\times 2^{2-\alpha}+21 }{\Gamma(3-\alpha)}
  +\frac{1}{5}\frac{ 3^{1-\alpha}-7\times 2^{1-\alpha}+21  }{\Gamma(2-\alpha)},
\end{split}
\end{equation*}
\begin{equation*}
\begin{split}
\omega^{(6)}_{3}
=&\frac{  4^{6-\alpha}-7\times 3^{6-\alpha}+21\times2^{6-\alpha}-35 }{\Gamma(7-\alpha)} + \frac{5}{2}\frac{ 4^{5-\alpha}-7\times 3^{5-\alpha}+21\times2^{5-\alpha} -35 }{\Gamma(6-\alpha)} +\frac{17}{6}\frac{ 4^{4-\alpha}-7\times 3^{4-\alpha}+21\times2^{4-\alpha}-35 }{\Gamma(5-\alpha)} \\
& +\frac{15}{8}\frac{ 4^{3-\alpha}-7\times 3^{3-\alpha} +21\times2^{3-\alpha} -35 }{\Gamma(4-\alpha)} +\frac{137}{180}\frac{ 4^{2-\alpha}-7\times 3^{2-\alpha}+21\times2^{2-\alpha} -35 }{\Gamma(3-\alpha)}+\frac{1}{6}\frac{ 4^{1-\alpha}-7\times 3^{1-\alpha}
  +21\times2^{1-\alpha} -35 }{\Gamma(2-\alpha)},
\end{split}
\end{equation*}
\begin{equation*}
\begin{split}
\omega^{(6)}_{4}
=&\frac{  5^{6-\alpha}-7\times 4^{6-\alpha}+21\times3^{6-\alpha} -35 \times2^{6-\alpha} +35  }{\Gamma(7-\alpha)}
  +\frac{5}{2}\frac{ 5^{5-\alpha}-7\times 4^{5-\alpha}+21\times3^{5-\alpha} -35 \times2^{5-\alpha} +35}{\Gamma(6-\alpha)}   \\
&  +\frac{17}{6}\frac{  5^{4-\alpha}-7\times 4^{4-\alpha}+21\times3^{4-\alpha} -35 \times2^{4-\alpha} +35 }{\Gamma(5-\alpha)}
  +\frac{15}{8}\frac{  5^{3-\alpha}-7\times 4^{3-\alpha}+21\times3^{3-\alpha} -35 \times2^{3-\alpha} +35 }{\Gamma(4-\alpha)}  \\
&  +\frac{137}{180}\frac{  5^{2-\alpha}-7\times 4^{2-\alpha}+21\times3^{2-\alpha} -35 \times2^{2-\alpha} +35 }{\Gamma(3-\alpha)}
 +\frac{1}{6}\frac{  5^{1-\alpha}-7\times 4^{1-\alpha}+21\times3^{1-\alpha} -35 \times2^{1-\alpha} +35 }{\Gamma(2-\alpha)},
\end{split}
\end{equation*}

\begin{equation*}
\begin{split}
\omega^{(6)}_{5}
 =& \frac{  6^{6-\alpha}-7\times 5^{6-\alpha}+21\times4^{6-\alpha} -35 \times3^{6-\alpha} +35\times2^{6-\alpha}-21    }{\Gamma(7-\alpha)}
  +\frac{5}{2}\frac{ 6^{5-\alpha}-7\times 5^{5-\alpha}+21\times4^{5-\alpha} -35\times3^{5-\alpha} +35\times2^{5-\alpha}-21 }{\Gamma(6-\alpha)}\\
&  +\frac{17}{6}\frac{ 6^{4-\alpha}-7\times 5^{4-\alpha}+21\times4^{4-\alpha}-35\times3^{4-\alpha} +35\times2^{4-\alpha}-21 }{\Gamma(5-\alpha)}
  +\frac{15}{8}\frac{ 6^{3-\alpha}-7\times 5^{3-\alpha}+21\times4^{3-\alpha}-35\times3^{3-\alpha} +35\times2^{3-\alpha}-21 }{\Gamma(4-\alpha)} \\
&  +\frac{137}{180}\frac{ 6^{2-\alpha}-7\times5^{2-\alpha}+21\times4^{2-\alpha}-35\times3^{2-\alpha}+35\times2^{2-\alpha}-21 }{\Gamma(3-\alpha)}
  +\frac{1}{6}\frac{ 6^{1-\alpha}-7\times 5^{1-\alpha}+21\times4^{1-\alpha}-35\times3^{1-\alpha}+35\times2^{1-\alpha}-21 }{\Gamma(2-\alpha)},
\end{split}
\end{equation*}

\begin{equation*}
\begin{split}
\omega^{(6)}_{j}
 =& \frac{ (j+1)^{6-\alpha} - 7j^{6-\alpha} + 21(j-1)^{6-\alpha} - 35(j-2)^{6-\alpha} + 35(j-3)^{6-\alpha} - 21(j-4)^{6-\alpha} + 7(j-5)^{6-\alpha} - (j-6)^{6-\alpha}  }{\Gamma(7-\alpha)}               \\
&+ \frac{5}{2}\frac{ (j+1)^{5-\alpha} - 7j^{5-\alpha} + 21(j-1)^{5-\alpha} - 35(j-2)^{5-\alpha} + 35(j-3)^{5-\alpha}-21(j-4)^{5-\alpha} + 7(j-5)^{5-\alpha} - (j-6)^{5-\alpha}  }{\Gamma(6-\alpha)}      \\
&+ \frac{17}{6}\frac{ (j+1)^{4-\alpha} - 7j^{4-\alpha} + 21(j-1)^{4-\alpha} - 35(j-2)^{4-\alpha} + 35(j-3)^{4-\alpha} - 21(j-4)^{4-\alpha} + 7(j-5)^{4-\alpha} - (j-6)^{4-\alpha} }{\Gamma(5-\alpha)}    \\
&+ \frac{15}{8}\frac{ (j+1)^{3-\alpha} - 7j^{3-\alpha} + 21(j-1)^{3-\alpha} - 35(j-2)^{3-\alpha} + 35(j-3)^{3-\alpha} - 21(j-4)^{3-\alpha} + 7(j-5)^{3-\alpha} - (j-6)^{3-\alpha} }{\Gamma(4-\alpha)}    \\
&+ \frac{137}{180}\frac{ (j+1)^{2-\alpha} - 7j^{2-\alpha} + 21(j-1)^{2-\alpha} - 35(j-2)^{2-\alpha} + 35(j-3)^{2-\alpha} - 21(j-4)^{2-\alpha} + 7(j-5)^{2-\alpha} - (j-6)^{2-\alpha} }{\Gamma(3-\alpha)} \\
&+ \frac{1}{6}\frac{ (j+1)^{1-\alpha} - 7j^{1-\alpha} + 21(j-1)^{1-\alpha} - 35(j-2)^{1-\alpha} + 35(j-3)^{1-\alpha} - 21(j-4)^{1-\alpha} + 7(j-5)^{1-\alpha} - (j-6)^{1-\alpha} }{\Gamma(2-\alpha)},~j\geq 6.
\end{split}
\end{equation*}
\end{itemize}
}}

\section*{Acknowledgments}This work was supported by NSFC 11601206, 11901266 and Natural Science Foundation of Gansu Province (No. 21JR7RA253).

\section*{Data availability}
I confirm I have included a data availability statement in my main manuscript file.

\section*{Declarations}

\subsection* {Funding: This work was supported by NSFC 11601206 and Hong Kong RGC grant (No. 25300818).}

\subsection* {Conflicts of interest/Competing interests: The authors have no conflicts of interest to declare that are relevant to the content of this article.}

\subsection* {Availability of data and material: Not applicable.}

\subsection* {Code availability: Not applicable.}

\subsection* {Authors' contributions: An equal contribution.}


\begin{thebibliography}{10}


\bibitem{ACYZ:21}
Akrivis G.,  Chen M.H., Yu, F.,   Zhou, Z.:
The energy technique for the six-step BDF method.
 SIAM J. Numer. Anal. \textbf{59},  2449-2472 (2021)

\bibitem{App:1984}
Appel K.I.:
The use of the computer in the proof of the four color theorem.
Proc. Amer. Philos. Soc. \textbf{128}, 35-39 (1984)

\bibitem{CLC:2015}
Cao J.X., Li C.P.,  Chen Y.Q.:
High-order approximation to Caputo derivatives and Caputo-type advection-diffusion equations (\textrm{\uppercase\expandafter{\romannumeral2}}).
Fract. Calc. Appl. Anal. \textbf{18}, 735-761 (2015)

\bibitem{BBR:2003}
Bhagat V., Bhattacharya R., Roy D.:
On the evaluation of generalized Bose-Einstein and Fermi-Dirac integrals.
Comput. Phys. Comm. \textbf{155}, 7-20 (2003)


\bibitem{CJB:2021}
Chen M.H., Jiang S.Z., Bu W.P.:
Two   $L1$ schemes on  graded meshes for  fractional Feynman-Kac equation.
J. Sci. Comput. \textbf{88}, 58 (2021)


\bibitem{CD:2015}
Chen M.H.,  Deng W.H.:
Discretized fractional substantial calculus.
ESAIM: Math. Model. Numer. Anal.   \textbf{49}, 373-394 (2015)


\bibitem{CD:2014}
Chen M.H.,  Deng W.H.:
Fourth order accurate scheme for the space fractional diffusion equations.
SIAM J. Numer. Anal.  \textbf{52},   1418-1438 (2014)

\bibitem{Flajolet:1999}
Flajolet P.:
Singularity analysis and asymptotics of Bernoulli sums.
Theoret. Comput. Sci. \textbf{215}, 371-381 (1999)


\bibitem{GSZ:2014}
Gao G.H., Sun Z.Z., Zhang H.W.:
A new fractional numerical differentiation formula to approximate the Caputo fractional derivative and its applications.
J. Comput. Phys. \textbf{259}, 33-50 (2014)



\bibitem{JLZ:2017}
Jin B., Li B.Y., Zhou Z.:
Correction of high-order {BDF} convolution quadrature for fractional evolution equations.
SIAM J. Sci. Comput. \textbf{39}, A3129-A3152 (2017)

\bibitem{JLZ:2016}
Jin B., Lazarov R., Zhou Z.:
An analysis of the {$L1$} scheme for the subdiffusion equation with nonsmooth data.
IMA J. Numer. Anal. \textbf{36}, 197-221 (2016)


\bibitem{Kopteva:2021}
Kopteva N.:
Error analysis of an $L2$-type method on graded meshes for a fractional-order parabolic problem.
Math. Comp. \textbf{90}, 19-40 (2021)


\bibitem{LeVeque:2007}
LeVeque R.J.:
Finite Difference Methods for Ordinary and Partial Differential Equations.
SIAM, Philadelphia (2007)

\bibitem{LCL:2016}
Li H.F., Cao J.X., Li C.P.:
High-order approximation to Caputo derivatives and Caputo-type advection-diffusion equations (\textrm{\uppercase\expandafter{\romannumeral3}}).
J. Comput. Appl. Math. \textbf{299}, 159-175 (2016)



\bibitem{LX:2007}
Lin Y.M., Xu C.J.:
Finite difference{/}spectral approximations for the time-fractional diffusion equation.
J. Comput. Phys. \textbf{225}, 1533-1552 (2007)


\bibitem{Lubich:1986}
Lubich Ch.:
Discretized fractional calculus.
SIAM J. Math. Anal. \textbf{17}, 704-719 (1986)


\bibitem{LST:1996}
Lubich Ch., Sloan I.H., Thom\'{e}e V.:
Nonsmooth data error estimates for approximations of an evolution equation with a positive-type memory term.
Math. Comp. \textbf{65}, 1-17 (1996)


\bibitem{LX:2016}
Lv C.W., Xu C.J.:
Error analysis of a high order method for time-fractional diffusion equations.
SIAM J. Sci. Comput. \textbf{38}, A2699-A2724 (2016)


\bibitem{MT:2004}
Meerschaert M.M., Tadjeran C.:
Finite difference approximations for fractional advection-dispersion flow equations.
J. Comput. Appl. Math. \textbf{172}, 65-77 (2004)


\bibitem{Metzler:00}
Metzler R., Klafter J.:
The random walk's guide to anomalous diffusion: a fractional dynamics approach.
Phys. Rep. \textbf{339}, 1--77 (2000)

\bibitem{Podlubny:1999}
Podlubny I.:
Fractional Differential Equations.
Academic Press, New York (1999)


\bibitem{SY:11}
Sakamoto K., Yamamoto M.:
Initial value/boundary value problems for fractional diffusion-wave equations and applications to some inverse problems.
J. Math. Anal. Appl. \textbf{382},  426--447 (2011)

\bibitem{STW:2011}
Shen J.,  Tang T.,  Wang L.:
Spectral Methods: Algorithms, Analysis and Applications.
Springer-Verlag, Berlin (2011)

\bibitem{SC:2020}
Shi J.K.,  Chen M.H.:
Correction of high-order BDF convolution quadrature for fractional Feynman-Kac equation with L\'{e}vy flight.
J. Sci. Comput. \textbf{85}, 28 (2020).



\bibitem{SOG:2017}
Stynes M., O'riordan E.,  Gracia J.L.:
Error analysis of a finite difference method on graded meshes for a time-fractional diffusion equation.
SIAM J. Numer. Anal. \textbf{55}, 1057-1079 (2017)


\bibitem{SW:2006}
Sun Z.Z., Wu X.N.:
A fully discrete difference scheme for a diffusion-wave system.
Appl. Numer. Math. \textbf{56}, 193-209 (2006)


\bibitem{Thomee:2006}
Thom\'{e}e V.:
Galerkin Finite Element Methods for Parabolic Problems.
Springer, New York (2006)

\bibitem{WYY:2020}
Wang Y.Y., Yan Y.B., Yang Y.:
Two high-order  time discretization schemes for subdiffusion problems with nonsmooth data.
Fract. Calc. Appl. Anal. \textbf{23}, 1349-1380 (2020)


\bibitem{YKF:2018}
Yan Y.B., Khan M., Ford N.J.:
An analysis of the modified {$L1$} scheme for time-fractional partial differential equations with nonsmooth data.
SIAM J. Numer. Anal. \textbf{56}, 210-227 (2018)

\end{thebibliography}
\end{document}